\documentclass[a4paper]{amsart}

\usepackage[utf8]{inputenc}
\usepackage[T1]{fontenc}
\usepackage[english]{babel}
\usepackage{amsfonts}
\usepackage{amsthm}
\usepackage{amssymb}
\usepackage{mathcomp}
\usepackage{mathrsfs}
\usepackage[all]{xy}
\usepackage{graphicx}
\usepackage{epstopdf}
\usepackage[centertags]{amsmath}
\usepackage{vmargin}
\usepackage{tikz}
\usepackage{tikz-cd}
\usepackage{dsfont}
\usepackage{cases}
\usepackage{shuffle}

\definecolor{BleuTresFonce}{rgb}{0.215, 0.215, 0.36}
\definecolor{airforceblue}{rgb}{0.36, 0.54, 0.66}

\usepackage[colorlinks,final,hyperindex]{hyperref}
\hypersetup{citecolor=black, linkcolor=black}

\theoremstyle{plain}
\newtheorem{thm}{Theorem}
\newtheorem*{thm*}{Theorem}
\newtheorem{lemma}{Lemma}
\newtheorem{prop}{Proposition}
\newtheorem*{prop*}{Proposition}
\newtheorem{cor}{Corollary}

\theoremstyle{definition}
\newtheorem{defin}{Definition}
\newtheorem*{nota}{\sc Notations}

\theoremstyle{remark}
\newtheorem{rmk}{\sc Remark}
\newtheorem{eg}{\sc Example}

\setpapersize{A4}
             
\newcommand{\mbs}{\mathbb{S}}
\newcommand{\mbk}{\mathbb{K}}
\newcommand{\mbn}{\mathbb{N}}

\newcommand{\PP}{\mathcal{P}}

\newcommand{\Tree}{\mathsf{Tree}}

\newcommand{\ov}{\overline}

\newcommand{\Aa}{\mathcal{A}}
\newcommand{\BB}{\mathcal{B}}
\newcommand{\CC}{\mathcal{C}}
\newcommand{\DD}{\mathcal{D}}
\newcommand{\EE}{\mathcal{E}}

\newcommand{\II}{\mathcal{I}}

\newcommand{\VV}{\mathcal{V}}
\newcommand{\WW}{\mathcal{W}}
\newcommand{\AAA}{\mathscr{A}}
\newcommand{\BBB}{\mathscr{B}}
\newcommand{\CCC}{\mathscr{C}}
\newcommand{\DDD}{\mathscr{D}}
\newcommand{\EEE}{\mathscr{E}}
\newcommand{\FFF}{\mathscr{F}}
\newcommand{\PPP}{\mathscr{P}}
\newcommand{\QQQ}{\mathscr{Q}}

\newcommand{\QQ}{\mathcal{Q}}

\newcommand{\ra}{\rightarrow}

\newcommand{\Tfree}{\mathbb{T}}

\newcommand{\Operad}{\mathsf{dg}-\mathsf{Operad}}
\newcommand{\cCoop}{\mathsf{cCoop}}
\newcommand{\gCoop}{\mathsf{gCoop}}
\newcommand{\dgCoop}{\mathsf{dgCoop}}

\newcommand{\gop}{\mathsf{gOperad}}

\newcommand{\dgMod}{\mathsf{dgMod}}
\newcommand{\gMod}{\mathsf{gMod}}
\newcommand{\catC}{\mathsf{C}}
\newcommand{\catD}{\mathsf{D}}

\newcommand{\colim}{\mathrm{colim}}
\newcommand{\Map}{\mathrm{Map}}

\newcommand{\Tw}{ Tw}
\newcommand{\End}{\mathrm{End}}

\newcommand{\itemt}{\item[$\triangleright$]}

\newcommand{\poubelle}[1]{}

\title{Algebraic operads up to homotopy}
\author{Brice Le Grignou}
\address{Universiteit Utrecht}
\email{b.n.legrignou@uu.nl}

\date{\today}

\begin{document}

\maketitle

\begin{abstract}
This paper deals with the homotopy theory of differential graded operads. We endow the Koszul dual category of curved conilpotent cooperads, where the notion of quasi-isomorphism barely makes sense, with a model category  structure Quillen equivalent to that of operads. This allows us to describe the homotopy properties of differential graded operads in a simpler and richer way, using obstruction methods.
\end{abstract}

\setcounter{tocdepth}{1}
\tableofcontents

\section*{Introduction}

In representation theory, algebras encode some types of endomorphisms on vector spaces, that is linear operations satisfying some relations. More generally, the notion of operads is a tool which governs multilinear operations. More specifically, an operad encodes a type of algebras like associative, commutative, Lie or Batalin--Vilkovisky algebras, in a way that a representation of this operad amounts to the data of a vector space together with a structure of algebra of that type. For instance, the representations of the operad called the Lie operad, see \cite[13.2]{LodayVallette12}, are vector spaces together with a Lie-algebra structure. The correspondence between operads and their types of algebras is functorial. Indeed, any morphism of operads $f: \PPP \ra \QQQ$ induces an adjunction between  the category of $\QQQ$-algebras and the category of $\PPP$-algebras.\\

This paper deals with the homotopy theory of differential graded operads over a field of characteristic zero. For any dg operad $\PPP$, the category of $\PPP$-algebras admits a projective model structure whose weak equivalences are quasi-isomorphisms and whose fibrations are surjections. Moreover, for any morphism of operads $f$ from $\PPP$ to $\QQQ$, the resulting adjunction between the category of $\PPP$-algebras and the category of $\QQQ$-algebras is a Quillen equivalence if and only if $f$ is a quasi-isomorphism on the underlying chain complexes of $\PPP$ and $\QQQ$. So, quasi-isomorphisms provide a suitable notion of equivalence of dg operads. We know that the category of dg operads carries a model structure whose weak equivalences are quasi-isomorphisms and whose fibrations are surjections, see  \cite{Hinich97}, \cite{Spitzweck01} and \cite{BergerMoerdijk03}.\\

Several issues appear when describing the homotopy theory of dg operads with this model structure. For instance, one can ask whether two dg operads are weakly equivalent; for example, whether a dg operad is formal, that is weakly equivalent to its homology. Moreover, how to describe in a concrete manner homotopies between morphism? This last issue is related to the computation of cofibrant resolutions of dg operads. A general tool to produce such cofibrant resolutions is provided by the operadic bar-cobar adjunction introduced first by Getzler and Jones \cite{GetzlerJones94} which relates augmented dg operads to dg cooperads. This is generalized in the article \cite{LeGrignou16} to any kind of dg operads. The absence of augmentation of a dg operad is encoded into a curvature at the level of cooperads. Thus, we have an adjunction relating the category of dg operads to the category of curved conilpotent cooperads.
 \[
\begin{tikzcd}
\text{Curved conilpotent cooperads} \arrow[r, shift left, "\Omega_u"] &  \text{dg Operads} \arrow[l, shift left, "B_c"] 
\end{tikzcd}
\]
The importance of this adjunction with respect to the computation of cofibrant operads lies in the fact that for any operad $\PPP$, the counit map $\Omega_u B_c \PPP \to \PPP$ is a cofibrant resolution of $\PPP$. So, to describe homotopies between morphisms of dg operads from $\PPP$ to $\QQQ$, it is convenient to take place in the larger framework of morphisms from $\Omega_u B_c \PPP$ to $\QQQ$, which are equivalent to morphisms of curved conilpotent cooperads from $B_c \PPP$ to $B_c \QQQ$; so it is convenient to encode the homotopy theory of dg operads not in the category of dg operads itself but in the category of curved conilpotent cooperads. This leads us to the following result.

\begin{thm*}
 There exists a model structure on the category of curved conilpotent cooperads whose cofibrations and weak equivalences are created by the cobar construction functor $\Omega_u$. Moreover, the adjunction $\Omega_u \dashv B_c $ is a Quillen equivalence.
\end{thm*}

This theorem generalises results of Lefevre-Hasegawa (\cite{LefevreHasegawa03}) and Positselski (\cite{Positselski11}) respectively about the homotopy theory of nonunital associative algebras and about the homotopy theory of unital associative algebras. The proof relies on the same kind of method initiated by Hinich. New difficulties appear with the combinatorics of trees and the interplay of symmetric groups.\\

Why switching from dg operads to curved conilpotent cooperads? First, all the objects of the model category of curved conilpotent cooperads are cofibrant. Then, the sub-category of fibrant curved conilpotent cooperads is equivalent to a category whose objects and morphisms are an homotopy loosening of respectively the notion of dg operads and the notion of morphisms of dg operads, that we call homotopy operads and $\infty$-morphisms. These new structures can be built on objects using obstruction methods. Moreover, it is a convenient framework to study formality of dg operads. Indeed, a dg operad $\PPP$ is formal if and only if there exists an $\infty$-morphism from $\PPP$ to its homology and whose first level map is a quasi-isomorphism. Finally, there exists a transfer theorem for homotopy operads as follows.

\begin{thm*}
 Let $f:\PPP \ra \QQQ$ be a morphism of dg-$\mbs$-modules which is both a surjection and a quasi-isomorphism. Suppose that $\PPP$ has a structure of homotopy operad. Then, there exists a structure of homotopy operad on $\QQQ$ and an extension of $f$ into an $\infty$-morphism of homotopy operads. 
\end{thm*}

One could think that dg operads are themselves algebras over a colored operad and apply the results of the article \cite{LeGrignou16} to get the present theorems. Actually, the actions of the symmetric groups underlying curved conilpotent cooperads seem to prevent them to be coalgebras over a colored cooperad.

\subsection*{Layout}

The article is organized as follows. The first section recalls the notions of operads, cooperads and the operadic bar-cobar adjunction. The second one recalls the Hinich model structure on dg operads and describes their cofibrations and a simplicial enrichment computing mapping spaces. The core of the article is the third part which establishes the model structure on curved conilpotent cooperads. The fourth section studies in details the fibrant objects of this model category which are a notion of operads up to homotopy that we call homotopy operads. The fifth section applies the formalism of homotopy operads to the study of algebras over an operad. Specifically, one interprets infinity-morphisms of algebras in terms of morphisms of homotopy operads.

\subsection*{Acknowledgements} This article is the third part of my PhD thesis. I would like to thank my advisor Bruno Vallette for his precious advice and careful review of this paper. I also would like to thank Damien Calaque and Kathryn Hess for reviewing my thesis. Also, the Laboratory J.A. Dieudonn\'e in the University of Nice provided excellent working conditions. Finally, I was supported by the ANR SAT until September 2016 and then by NWO.

\subsection*{Preliminaries}

\begin{itemize}
 \itemt We work over a field $\mbk$ of characteristic zero.
 \itemt The category of graded $\mbk$-modules is denoted $\gMod$. The category of chain complexes is denoted $\dgMod$. These two categories are endowed with their usual closed symmetric monoidal structure. The internal hom is denoted by $[\ ,\ ]$. The category of chain complexes is also endowed with its projective model structure whose weak equivalences are quasi-isomorphisms and whose fibrations are degreewise surjections. The degree of an homogeneous element $x$ of a graded $\mbk$-module or a chain complex is denoted by $|x|$.
 \itemt For any integer $n$, let $D^n$ be the chain complex generated by one element in degree $n$ and its boundary in degree $n-1$. Let $S^n$ be the chain complex generated by a cycle in degree $n$.
  \itemt The following type of diagram
 \[
\begin{tikzcd}
 \catC \arrow[r, shift left, "L"]  & \catD \arrow[l, shift left, "R"] 
\end{tikzcd}
\]
means that the functor $R$ is right adjoint to $L$.
\itemt Let $(F_n X)_{n \in \mbn}$ be a filtration on a chain complex or a graded $\mbk$-module $X$. We denote by $G_n X$ the quotient $F_nX / F_{n-1} X$ and $GX := \bigoplus_n G_nX$.
\end{itemize}

The following theorem will be of major use.

\begin{thm}[Maschke]
When the characteristic of the field $\mbk$ is zero, any module over the ring $\mbk[\mbs_n]$ is projective and injective.
\end{thm}
\vspace{1cm}

\section{Operads and cooperads}

In this first section, we recall the notions of operads and cooperads. We refer to \cite{LodayVallette12} for more details. Moreover, we show that the category of dg operads and the category of curved conilpotent cooperads are presentable. Finally, we recall the refined bar-cobar adjunction introduced in \cite{LeGrignou16}.

\subsection{Symmetric modules, operads and cooperads}

\begin{defin}[Symmetric modules]\leavevmode
Let $\mbs$ be the groupoid whose objects are the integers $\mbn$ and whose morphisms are
 $$
\begin{cases}
 \hom_\mbs(n,m)=  \emptyset, \text{ if }n \neq m,\\
  \hom_\mbs(n,n)= \mbs_n, \text{ otherwise.}
\end{cases}
$$
A \textit{graded} $\mbs$-\textit{module} (resp \textit{dg} $\mbs$-\textit{module}) is a contravariant functor from the category $\mbs$ to the category $\gMod$ of graded $\mbk$-modules (resp. the category $\dgMod$ of chain complexes).
\end{defin}  

The category of $\mbs$-modules may be endowed with the Day monoidal product
\[
	\VV \circledast \WW (n) = \bigoplus_{p+q=n} (\VV (p) \otimes \WW(q)) \otimes_{\mbs_p \times \mbs_q} \mbs_n .
\]
This give us a symmetric monoidal structure whose unit is $(\mbk, 0, 0, \ldots)$.
Then, we can endow the category of $\mbs$-modules with the composition monoidal structure given for any $\mbs$-modules $\VV$ and $\WW$
by the formula
\[
	(\VV \circ \WW) = \bigoplus_n \VV(n) \otimes_{\mbs_n} \WW^{\circledast n} . 
\]
This means that 
$$
(\VV \circ \WW) (0):= \VV (0) \oplus \bigoplus_{k \geq 1 } \VV(k) \otimes_{\mbs_k} (\WW(0) \otimes \cdots \otimes \WW(0))\ ,
$$
and for any integer $n \geq 1$,
$$
(\VV \circ \WW) (n):= \bigoplus_{X_1 \sqcup \cdots \sqcup X_k = \{1, \ldots, n\}\atop k \geq 1 } \VV(k) \otimes_{\mbs_k} ((\WW( \# X_1) \otimes \cdots \otimes \WW( \# X_k))\otimes_{\mbs_{\#X_1} \times \cdots \times \mbs_{\#X_k}}
 \mbs_n) \ ,
$$
where the coproduct is over the ordered partitions of the set $\{1, \cdots, n\}$, that is, the $k$-tuples of subsets of $\{1, \cdots, n\}$ with empty intersections and  union $\{1, \cdots, n\}$. Moreover, $\# X_i$ is the cardinal of the set $X_i$. The monoidal unit is the $\mbs$-module $\II$ which is $\mbk$ in arity $1$ and $\{0\}$ in other arities.
  
\begin{nota}
   Let $f: \VV \ra \VV'$, $g: \WW \ra \WW'$ and $h: \WW \ra \WW'$ be maps between $\mbs$-modules. Then, we denote by $f \circ(g;h)$ the map from $\VV \circ \WW$ to $\VV' \circ \WW'$ defined as follows.
$$
f \circ (g;h) :=\sum_{i+j=n-1}  f \otimes_{\mbs_n} (g^{\circledast i} \circledast h \circledast g^{\circledast j})\ .
$$
In the case where $g=Id$, we use the following notation.
$$
f\circ' h := f \circ (Id;h)\ .
$$
\end{nota}

\begin{nota}
 For any two graded $\mbs$-modules (resp. dg $\mbs$-modules) $\VV$ and $\WW$, we denote by $[\VV,\WW]$ the graded $\mbk$-module (resp. chain complex):
 $$
 [\VV,\WW]_n := \prod_{k \geq 0\atop l \in \mbn} \hom_{\mbk[\mbs_n]}(\VV(k)_l , \WW(k)_{l+n})\ .
 $$
 In that context morphisms of chain complexes from $X$ to $[\VV,\WW]$ are in one-to-one correspondence with morphisms of $\mbs$-modules from $X \otimes \VV$ to $\WW$.
\end{nota}

\begin{defin}[Operads]
A \textit{graded operad} (resp. \textit{dg operad}) is a monoid $\PPP:=(\PP, \gamma, \upsilon)$ in the category of graded $\mbs$-modules (resp. dg $\mbs$-modules). We denote by $\gop$ (resp. $\Operad$) the category of graded operads (resp. dg operads).
\end{defin}

A degree $k$ derivation $d$ on a graded operad $\PPP=(\PP, \gamma,\upsilon)$ consists of degree $k$ maps $d: \PP(n) \ra \PP(n)$ which commute with the actions of $\mbs_n$ and such that
$$
d\ \gamma = \gamma\  (d \circ Id + Id \circ' d)\ .
$$

\begin{defin}[Cooperads]
A \textit{graded cooperad} (resp. \textit{dg cooperad}) is a comonoid $\CCC:=(\CC, \Delta, \epsilon)$ in the category of graded $\mbs$-modules (resp. dg $\mbs$-modules). We denote by $\ov \CC$ the kernel of the morphism $\epsilon : \CC \ra \II$. We denote by $\gCoop$ (resp. $\dgCoop$) the category of graded cooperads (resp. dg cooperads). A cooperad $\CCC$ is said to be coaugmented if it is equipped with a morphism of cooperads $\II \ra \CCC$. In this case, we denote by $1$ the image of the unit of $\mbk$ into $\CC(1)$.
\end{defin}

A degree $k$ coderivation on a cooperad $\CCC=(\CC, \Delta, \epsilon)$ is a degree $k$ map $d$ of $\mbs$-modules from $\CC$ to $\CC$ such that 
$$
\Delta\  d =(d \circ Id + Id \circ' d)\Delta \ .
$$
 If the cooperad is coaugmented, we also require that $d(1)=0$.
 
 \begin{nota}
 Let $(\CC, \Delta,\epsilon, 1)$ be a coaugmented cooperad. Then we denote by $\ov \Delta$ the map from $\ov\CC$ to $\ov\CC \circ \CC$ defined by
 \[
 \ov\Delta = \Delta - 1 \circ Id - Id \circ 1\ .
 \]
 Moreover, we denote by $\Delta_2$ the map from $\ov \CC$ to $\CC \circ \CC$ defined by
$$
\Delta_2 :=  Id \circ (\epsilon ; Id)\ov \Delta\ .
$$
Sometimes, $\Delta_2$ will be extended to all $\CC$ by
\[
\Delta_2 (1) =1 \otimes 1\ .
\]
\end{nota}
 
\begin{defin}[Curved cooperads]
 A curved cooperad is a coaugmented graded cooperad $\CCC=(\CC, \Delta, \epsilon, 1)$ equipped with a degree $-2$ map $\theta : \CC(1) \ra \mbk$ and a degree $-1$ coderivation $d$ such that
$$
\begin{cases}
d^2 = \big(\theta \otimes Id-  Id \otimes \theta \big)  \Delta_2\ ,\\
\theta d =0 \ .
\end{cases}
$$
\end{defin}

\begin{eg}\label{egend}
 Let $\VV$ and $\WW$ be two chain complexes (resp. graded $\mbk$-modules). Let $\End^\VV_\WW$ be the following $\mbs$-module
 \[
 \End^\VV_\WW (n) := [\VV^{\otimes n}, \WW]\ ,
 \]
 where the action of $\mbs_n$ is given by permuting the inputs $\VV^{\otimes n}$. Moreover, we denote
 \[
 \End_\VV := \End^\VV_\VV\ .
 \]
 The composition of multi-linear maps of $\VV$ induces a structure of operad on $\End_\VV$. Moreover, for any morphism of chain complexes $f: \VV \to \WW$, consider the following pullback of $\mbs$-modules
 \[
 \xymatrix{\End_\VV \times_{\End^\VV_\WW} \End_\WW \ar[r] \ar[d] & \End_\VV \ar[d]^{f \circ } \\
 \End_\WW \ar[r]_{\circ f^{\otimes n}}& \End^\VV_\WW\ .}
 \]
 Then, the operad structures on $\End_\VV$ and on $\End_\WW$ induce an operad structure on this pullback.
\end{eg}

\subsection{The tree module}

\begin{defin}[Trees]
 A tree $(\mathrm{edges}(T),<, \mathrm{vert}(T), \mathrm{leaves}(T))$, is the data of a non empty finite poset called the poset of edges $(\mathrm{edges}(T),<)$ so that
\begin{itemize}
 \itemt every subset has a lower limit;
 \itemt for edge $e$, the sub-poset $\{x\leq e\}$ is a linear poset;
\end{itemize}
and partition of the set of edges
\[
	\mathrm{edges}(T)= \mathrm{vert}(T)\sqcup \mathrm{leaves}(T)
\]
where all elements of $\mathrm{leaves}(T)$ are maximal. Moreover, for any $v \in \mathrm{vert}(T)$, the inputs of $v$ are the edges $e>v$ so that there is no
edge $e'$ so that $e>e'>v$. The number of inputs of a vertex $v$ is denoted $\# v$.We have in particular a smallest element called the root.
An isomorphism of trees is an isomorphism of posets that preserves the partition leaves-vertices. This gives is the groupoid of trees $\mathsf{Trees}$.
\end{defin}

\begin{defin}[Tree module] Let $\VV$ be an $\mbs$-module (graded or differential graded) and let $T$ be a tree with $p$ vertices and $q$ leaves.
For any vertex $v$, let $\VV(v)$ be the right module over the automorphism group of inputs of $v$
\[
	\VV(v)= \VV(\#v) \otimes_{\mbs_{\# v}} \mathrm{Bij}(\mathrm{input}(v), \{1, \cdots ,\# v\}) .
\]
Then, let $\bigotimes_T \VV$ be the following graded $\mbk$-module (resp. chain complex)
$$
\bigotimes_T \VV:= \big(\bigoplus_{v_1 ,\ldots, v_p} \VV(v_1)\otimes \cdots \otimes \VV(v_p) \big)_{\mbs_p}\ ,
$$
where the sum is taken over the bijections from the set $\{1,\cdots , p \}$ to the set of vertices of $T$.
The formula $\bigotimes_T -$ defines a contravariant functor from trees to graded $\mbk$-module (resp. chain complexes). 
Besides, let $T(\VV)$ be the $\mbs$-module such that $T(\VV)(k)=0$ for $q \neq k$ and
$$
T(\VV)(q):= \int_{\mathrm{Aut}(T)} \mathrm{Bij}(\{1, \ldots, q\}, \mathrm{leaves}(T))\otimes \bigotimes_T \VV\ ,
$$
where Bij stands fro bijections.
Finally, let $\Tfree \VV$ be the following $\mbs$-module called the tree module
$$
\Tfree \VV := \bigoplus_{[T]} T(\VV)\ ,
$$
where the sum is taken over the isomorphism classes of trees. Moreover, the reduced tree module is
$$
\overline{\Tfree} \VV := \bigoplus_{[T]\neq |} T(\VV)\ ,
$$
where the sum is taken over the isomorphism classes of non trivial trees.
\end{defin}

\begin{nota}
 Thus, an element of $T(\VV)(q)$ may be represented by a a tensor product
 \[
 	\phi \otimes x_1 \otimes \cdots x_p
 \]
 for $\phi$ a bijection from the leaves of $T$ to $\{1, \ldots, q\}$, and where $x_1, \ldots,x_p$ are elements of $\VV$. We will usually omit $\phi$ in the formula, since it is unchanged under all the constructions performed.
\end{nota}

\begin{defin}[Induced filtration on trees]\label{defininducedfilt}
 Given a graded or dg $\mbs$-module $\VV$ equipped with a filtration
 \[
 	F_0 \VV \hookrightarrow F_1 \VV \hookrightarrow  F_2 \VV \hookrightarrow \cdots \hookrightarrow F_n \VV \hookrightarrow \cdots
 \]
 one can endow $\bigotimes_T \VV$ with a filtration for any tree $T$ with $q$ vertices
 \[
 	F_n (\bigotimes_T \VV) = \sum_{p_1+\cdots +p_n\leq n}(F_{p_1} \VV(\# v_1))\otimes \cdots \otimes (F_{p_q}\VV(\#v_q)) 
	\subset \VV(\# v_1)\otimes \cdots \otimes \VV(\#v_q).
 \]
 Then, one can define a filtration on the $\mbs$-modules $T(\VV)$ and $\Tfree (\VV)$
\begin{align*}
 	F_n(T(\VV))(q)&:= \int_{\mathrm{Aut}(T)} \mathrm{Bij}(\{1, \ldots, q\}, \mathrm{leaves}(T))\otimes F_n(\bigotimes_T \VV)\ ,
	\\
	F_n(\Tfree \VV) &:= \bigoplus_{[T]} F_n(T(\VV)).
\end{align*}
\end{defin}

\begin{nota}\leavevmode
\begin{itemize}
\itemt For any $\mbs$-module $\VV$, we denote by $\pi_\VV$ the canonical projection of $\Tfree \VV$ onto $\VV$.
 \itemt  We denote by $\Tfree^{\leq n} \VV$ (resp. $\Tfree^{n} \VV$) the sub-$\mbs$-module of $\Tfree \VV$ made up of trees with $n$ or less than $n$ vertices (resp. with $n$ vertices).
  \itemt  We denote by $\Tfree^{h\leq n} \VV$ (resp. $\ov{\Tfree}^{h\leq n} \VV$) the sub-$\mbs$-module of $\Tfree \VV$ (resp. $\ov{\Tfree}\VV$) made up of trees (resp. non trivial trees) of height equal or lower than $n$.
 \itemt Let $T$ be a tree. We denote by $\{T= T_1 \sqcup \ldots \sqcup T_k \}$ the partition of $T$ by $k$ subtrees.
 \itemt Let $f: \Tfree \VV \ra \WW$ be a morphism of $\mbs$-modules. We denote by $f(T)$ the restriction of $f$ on $T(\VV) \subset \Tfree \VV$. Moreover, if the tree $T$ decomposes into a partition of subtrees $T=T_1 \sqcup \cdots \sqcup T_k$, then we denote by  
 $$
 f(T_1) \otimes \cdots \otimes f(T_k)
 $$
 the map from $T(\VV)$ to $T/T_1, \ldots, T_k (\WW)$ which consists in applying $f(T_i)$ on any subtree $T_i$ of $T$.
\end{itemize}
\end{nota}

The definition of the tree module from \cite{LodayVallette12} is different from ours. Let us build the bridge that links the two definitions. A refined version of this link will appear in the paper \cite{LeGrignouLejay19b}.

Let us choose, a tree for any isomorphism class of trees. Thus, from now on, by a "tree", we mean one these chosen trees. Subsequently, if two trees are isomorphic they are actually equal and we have a canonical isomorphism between them, that is the identity.

Let $T$ be a non trivial tree, whose root vertex $v$ has $k>0$ inputs. Then, we can decompose $T$ as $(v;T_1, \ldots, T_k) $ where $T_1, \ldots, T_k$ are trees whose root is an input of $v$ and which contain all the edges of $T$ above this input.
Let $n$ be the number of leaves of $t$.

\begin{lemma}
Let $n_i$ be the number of leaves of $T_i$ for each $1 \leq i \leq k$.
As left $(\prod_{i=1}^k \mathrm{Aut} (T_i)) \times\mbs^{\mathrm{op}}_n$-modules in sets, we have
\[
	\mathrm{Bij}(\{1, \ldots, n\}, \mathrm{leaves} (T)) = \left(\prod_{i=1}^k
	\mathrm{Bij}(\{1, \ldots, n_i\}, \mathrm{leaves} (T_i))\right)
	\times_{\prod_i
	\mbs_{n_i}} \mbs_n.
\]
\end{lemma}

\begin{proof}
This isomorphism sends $((\phi_1, \ldots, \phi_k), \sigma)$
to the bijection
\[
	\{1, \ldots ,n\}  \xrightarrow{\sigma}
	\{1, \ldots ,n\}
	= \coprod_i \{1, \ldots ,n_i\}
	\xrightarrow{\coprod_i  \phi_i}
	\coprod_i \mathrm{leaves} (T_i)
	= \mathrm{leaves} (T) .
\]
\end{proof}

\begin{lemma}
We have an isomorphism of $\mbs$-modules
\[
	\left(\bigotimes_{i=1}^k \bigotimes_{T_i} \VV
	\right)\otimes_{\prod_i \mathrm{Aut}(T_i)}
	\mathrm{Bij}(\{1, \ldots, n\}, \mathrm{leaves} (T))
	= T_1(\VV) \circledast \cdots \circledast T_k(\VV).
\]
\end{lemma}

\begin{proof}
One has
\begin{align*}
	&\left(\bigotimes_{i=1}^k \bigotimes_{T_i} \VV
	\right)\bigotimes_{\prod_i \mathrm{Aut}(T_i)}
	\mathrm{Bij}(\{1, \ldots, n\}, \mathrm{leaves} (T))
	\\
	&= \left(\bigotimes_{i=1}^k \bigotimes_{T_i} \VV\right)
	\bigotimes_{\prod_i \mathrm{Aut}(T_i)}
	\left(\prod_{i=1}^k
	\mathrm{Bij}(\{1, \ldots, n_i\}, \mathrm{leaves} (T_i))\right)
	\times_{\prod_i
	\mbs_{n_i}} \mbs_n
	\\
	&= \bigotimes_{i=1}^k \left(\bigotimes_{T_i} \VV
	\otimes_{\mathrm{Aut}(T_i)} \mathrm{Bij}(\{1, \ldots, n_i\}, \mathrm{leaves} (T_i))
	\right)
	\bigotimes_{\prod_i \mbs_{n_i}} \mbs_n
	\\
	&= \left(\bigotimes_{i=1}^k  T_i(\VV)\right)
	\bigotimes_{\prod_i \mbs_{n_i}} \mbs_n
	\\
	&= T_1(\VV) \circledast \cdots \circledast T_k(\VV).
\end{align*}
\end{proof}

Suppose also that the trees $T_1, \ldots,T_k$ are counted in a way that isomorphic trees are gathered. This means that for any $1 \leq i <j<l\leq k$, if $T_i$ is isomorphic to $T_l$, then they are also isomorphic to $T_j$. Then, we can reindex the trees $T_1, \ldots, T_k$ as
\[
	(T_1, \ldots, T_k)= (T_{1,1}, \ldots, T_{1, k_1}, T_{2,1}, \ldots,T_{2,k_2}, \ldots, T_{p,1} \ldots, T_{p,k_p}) ,
\]
so that two trees $T_{i,j}$ and $T_{i,l}$ are isomorphic while two trees $T_{i,j}$ and $T_{i',l}$ are not isomorphic when $i\neq i'$.

\begin{lemma}
We have a canonical isomorphism of $\mbs$-modules
\[
	      	T(\VV) =\VV(k) \otimes_{\mbs_k}
	      	\left( \bigoplus_{\sigma}  T_{\sigma^{-1}(1)}(\VV)
	      	\circledast \cdots \circledast T_{\sigma^{-1}(k)}(\VV)
	      	\right),
\]
where the coproduct is taken over the $(k_1,\ldots,k_p)$-shuffle
permutations.
\end{lemma}

\begin{proof}
We have
\[
	\bigotimes_T\VV\otimes \mathrm{Bij}(\{1, \ldots, n\}, \mathrm{leaves}(T))
	= \left(\VV(k) \otimes \bigotimes_{i=1}^k  \bigotimes_{T_i}\VV \right)
	\otimes \mathrm{Bij}(\{1, \ldots, n\}, \mathrm{leaves}(T))
\]
Besides, $\mathrm{Aut} (T)$ is a semi-direct product of $\prod_i \mathrm{Aut}(T_i)$ with $\prod_{j=1}^p \mbs_{k_j}$. Then, to quotient the object above by
$\mathrm{Aut}(T)$, one can first  quotient by the normal subgroup $\prod_i \mathrm{Aut}(T_i)$
\[
	\bigotimes_T\VV\otimes_{\prod_i \mathrm{Aut}(T_i)} \mathrm{Bij}(\{1, \ldots, n\}, \mathrm{leaves}(T))
	= \VV(k) \otimes  \left( T_1(\VV) \circledast \cdots \circledast
	T_k(\VV)\right)
\]
Finally,
\begin{align*}
	T(\VV) &= \VV(k) \otimes_{\prod_{j=1}^p \mbs_{k_j} }
	\left( T_1(\VV) \circledast \cdots \circledast
	T_k(\VV)\right)
	\\
	&= \VV(k) \otimes_{\mbs_k} \mbs_k \otimes_{\prod_{j=1}^p \mbs_{k_j} }
	\left( 
	  T_1(\VV) \circledast \cdots \circledast
	T_k(\VV) \right)
	\\
	&= \VV(k) \otimes_{\mbs_k}
	\left( \bigoplus_{\sigma \in \shuffle (k_1, \cdots , k_p)}
	T_{\sigma^{-1}(1)}(\VV) \circledast \cdots \circledast
	T_{\sigma^{-1} (k)}(\VV)\right) .
\end{align*}
where the symbol $\shuffle (k_1, \cdots , k_p)$ denotes the set of $(k_1, \cdots , k_p)$-shuffle permutations.
\end{proof}

\begin{lemma}\label{propcanonicaliso}
For any natural integer $n$, one has
\[
	{\ov{\Tfree}^{h \leq n+1}}\VV = \VV \circ 
	{\Tfree}^{h \leq n}  \VV.
\]
\end{lemma}

\begin{proof}
For a non trivial tree $T$ with height equal or lower than $n+1$ and whose root has $k$ inputs, let
$(T_1, \dots, T_k)$ be the subtrees above the root where isomorphic subtrees are gathered together.
The coproduct $\bigoplus_\sigma T_{\sigma^{-1}(1)}(\VV)
\circledast \dots \circledast T_{\sigma^{-1}(k)}(\VV)$ taken over the
$(k_1, \dots, k_p)$-shuffle permutations does not depend on the
choice of ordering (provided isomorphic subtrees are gathered together). We have
\[ 
	\bigoplus_{T} \bigoplus_\sigma
	T_{\sigma^{-1}(1)}(\VV)
	\circledast \cdots \circledast
	T_{\sigma^{-1}(k)}(\VV)
	= \bigoplus_{T_1,\ldots,T_k}
	T_1(M) \circledast \cdots \circledast T_k(M)
\]
where the coproduct on the left is taken over all the non trivial trees with height equal or lower than $n+1$ and whose root vertex has $k$-inputs
and over the set of shuffle permutations, and where the 
the coproduct on the right is taken over all trees with height equal or lower than $n$.

One then has
\begin{align*}
	{\ov{\Tfree}^{h \leq n+1}}\VV
	&= \VV(0)\bigoplus_{k >0}
	\bigoplus_{T}
	\VV(k) \otimes_{\mbs_k}
	\left(\bigoplus_\sigma T_{\sigma^{-1}(1)}(\VV)
	\circledast \cdots \circledast
	T_{\sigma^{-1}(k)}(\VV) \right)
	\\
	&= \VV(0)
	\bigoplus_{k>0}  \VV(k) \otimes_{\mbs_k}
	\left( \bigoplus_{T}
	\bigoplus_\sigma
	T_{\sigma^{-1}(1)}(\VV)
	\circledast \cdots \circledast
	T_{\sigma^{-1}(k)}(\VV)
	\right)
	\\
	&= \VV(0)
	\bigoplus_{k>0}  \VV(k) \otimes_{\mbs_k}
	\left( \bigoplus_{T_1, \ldots, T_k}
	T_1(\VV)\circledast \cdots \circledast T_k(\VV)\right)
	 \\
	&=
	 \bigoplus_  \VV(k) \otimes_{\mbs_k}
	\left(\Tfree^{h \leq n} \VV\right)^{\circledast k}
	\\
	&=
	\VV \circ \Tfree^{h \leq n}  \VV .
\end{align*}
\end{proof}

The result of Lemma \ref{propcanonicaliso} was used in the book \cite{LodayVallette12} as a definition of a tree module.

\begin{prop}\cite[\S 5.5]{LodayVallette12}
For any $\mbs$-module $\VV$, the tree module $\Tfree \VV$ has the structure of an operad given by the grafting of trees. Moreover, the functor $\Tfree$ from the category of $\mbs$-modules to the category of operads is left adjoint to the forgetful functor.
\end{prop}

This structure of an operad may be defined by induction with Lemma \ref{propcanonicaliso}.

There is a one-to-one correspondence between the degree $k$ derivations on the graded free operad $\Tfree \VV$ and the degree $k$ maps from $\VV$ to $\Tfree \VV$. Indeed, from such a map $u$ one can produce the derivation $D_u$ such that
$$
D_u (p_1 \otimes \cdots \otimes p_n) := \sum (-1)^{k(|p_1| + \cdots + |p_{i-1}|)} p_1 \otimes \cdots \otimes u(p_i) \otimes \cdots \otimes p_n\ . 
$$

\begin{defin}[Alternated tree module]
 Let $\VV$ and $\WW$ be two $\mbs$-modules and let $T$ be a tree. The alternated tree module $T (\VV,\WW)$ is the sub $\mbs$-module of $T (\VV \oplus \WW)$ made up of labellings of the tree $T$ such that, if a vertex is labelled by an element of $\VV$ (resp. $\WW$), then its neighbours are labelled by $\WW$ (resp. $\VV$). Moreover, $\Tfree (\VV , \WW) $ is the sum over the isomorphism classes of trees $T$ of $T (\VV , \WW)$.
\end{defin}

\begin{prop}
 Let $f:\VV \ra \VV'$ and $g: \WW \ra \WW'$ be two quasi-isomorphisms of dg-$\mbs$-modules and let $T$ be a tree. Then, the map $T(f):T(\VV) \ra T(\VV')$ and the map $T(f,g): T(\VV,\WW) \ra T(\VV',\WW') $ are quasi-isomorphisms.
\end{prop}

\begin{proof}
It follows from the definition of the tree module and the Kunneth formula.
\end{proof}

\subsection{Conilpotent cooperads}

\begin{prop}\cite[\S 5.8]{LodayVallette12}
For any graded vector space (resp. chain complex) $\VV$, the tree module $\Tfree(\VV)$ has the structure of a coaugmented cooperad given by degrafting trees. This defines a functor $\Tfree^c$ from graded vector space (resp. chain complexes) to coaugmented cooperads in graded vector space (resp. chain complexes).
\end{prop}

\begin{defin}
 A conilpotent cooperad (in chain complexes or graded $\mbk$-modules) is a coaugmented cooperad $\CCC$ so that there exists a morphism of coaugmented cooperads $\CCC\to \Tfree^c \VV$ that is a degreewise monomorphism.
\end{defin}

\begin{prop}\cite[\S 5.8]{LodayVallette12}
The functor $\Tfree^c$ from the category of $\mbs$-modules to the category of conilpotent cooperads is right adjoint to the functor $\CCC \mapsto \ov \CC$. 
\end{prop}

\begin{nota}
 For any conilpotent cooperad $\CCC$, we denote $\delta_\CCC$ the canonical morphism of conilpotent cooperads
 \[
 	\CCC \to \Tfree^c\overline{\CC},
 \]
 induced by the adjunction between $\Tfree^c$ and the forgetful functor.
\end{nota}

\begin{rmk}
Actually, the reduced tree module $\overline \Tfree$ is a comonad on the category of $\mbs$-modules and 
the mapping
\[
	\CCC \to \CCC \oplus \mbk
\]
induces an equivalence of categories from $\overline \Tfree{}$-coalgebras to conilpotent cooperads. 
\end{rmk}

\begin{defin}
 The category of curved conilpotent cooperads is denoted $\cCoop$.
\end{defin}

 \begin{defin}[Coradical filtration]\label{definicorad}
Let $\CCC=(\CC, \Delta, \epsilon,1)$ be a conilpotent cooperad. The coradical filtration $(F_n \CCC)_{n=0}^\infty$ of $\CCC$ is defined as follows
$$
F_n \CCC (k ):= \{x \in \CC(k) |\delta_\CCC (x) \in \Tfree^{\leq n} \ov \CC\}\ .
$$
It has the following property (see \cite{LeGrignou16})
$$
 \Delta (F^{rad}_n \CCC) \subset \sum_{k \in \mbn \atop p_0 + \cdots + p_k \leq n}
 (F^{rad}_{p_0} \CCC)(k) \otimes_{\mbs_k} (F^{rad}_{p_1} \CCC  \circledast \cdots \circledast F^{rad}_{p_k} \CCC )\ .
 $$
\end{defin}

There is a one-to-one correspondence between degree $k$ coderivations on $\Tfree^c (\VV)$ and degree $k$ maps from $\ov \Tfree (\VV)$ to $\VV$. Indeed, from such a map $u$ one can produce the coderivation $D_u$ such that
$$
D_u (T) := \sum_{T' \subset T} Id \otimes \cdots \otimes u(T') \otimes \cdots \otimes Id\ . 
$$
Note that the same statement holds for the sub-cooperads of $\Tfree^c \VV$ of the form $\Tfree^{\leq n} \VV$ for any integer $n$.

\subsection{Presentability}

In this section, we prove that both the category of differential graded operads and the category of curved conilpotent cooperads are presentable.

\begin{prop}
 The category $\Operad$ of differential graded operads is presentable.
\end{prop}

\begin{proof}
The tree module endofunctor $\Tfree$ of the category of $\mbs$-modules is a monad and the category of operads is monadic over this monad. Moreover, the functor $\Tfree$ preserves filtered colimits and so is accessible. We conclude by \cite[Theorem~2.78]{AdamekRosicky94}.
\end{proof}

We will now prove that the category of curved conilpotent cooperads is presentable. The essence of this result is that any cooperad is the colimit of the filtered diagram of its finite dimensional sub-cooperads.

\begin{prop}\label{prescoop}
 The category $\cCoop$ of curved conilpotent cooperads is presentable.
\end{prop}

\begin{lemma}\label{lemma:coopcolim}
The category of curved conilpotent cooperads is cocomplete. The forgetful functor from the category of curved conilpotent cooperads to the category of graded conilpotent cooperads creates colimits.
\end{lemma}

\begin{proof}
Straightforward.
\end{proof}

\begin{defin}
 We say that a graded $\mbs$-module $\VV$ is finite dimensional if the $\mbk$-module $\bigoplus_{n,m \in \mbn} \VV(n)_m$ is finite dimensional. We say that $\VV$ is aritywise finite dimensional if $\bigoplus_{m\in \mbn} \VV(n)_m$ is of finite dimension for any integer $n\in \mbn$.
\end{defin}

\begin{prop}\label{prop:coopfinite}
For any graded conilpotent cooperad $\CCC=(\CC,\Delta,\epsilon)$ and any element $x \in \CC(k)$, there exists a finite dimensional sub-cooperad of $\CCC$ which contains $x$.
\end{prop}

\begin{proof}
 Let us prove the proposition by induction on the coradical filtration. If, $x \in F_0^{rad}\CCC$, then $x$ belongs to finite dimensional sub-cooperad which is just
 $F_0^{rad}\CCC=\II$. Suppose that any $x' \in F_n^{rad}\CCC$ is contained in a finite dimensional sub-cooperad and let $x\in F_{n+1}^{rad}\CCC$. The decomposition $\Delta(x)$ is generated by elements of the form
 \[
 	\Delta (x) = x \otimes 1 + 1 \otimes + \sum_n \sum y_{(0)} \otimes ((y_{(1)} \otimes \cdots \otimes y_{(n)} \otimes \sigma)
\]
where the elements $y_{(i)}$ are of finite number and all belong to $F_{n}^{rad}\CCC$. Thus let $\DDD$ be a finite dimensional sub-cooperad that contains all the elements $y_{(i)}$. The smallest sub $\mbs$-module of $\CC$ that contains $\DD$ and $x$ is finite dimensional and is a sub-cooperad.
\end{proof}

\begin{rmk}
 Aubry and Chataur (\cite{AubryChataur03}) prove a similar result in the case of nonnegatively graded $\mbk$-modules (actually for nonnegatively graded chain complexes) and general cooperads.
\end{rmk}

\begin{cor}\label{cor:pres}
For any curved conilpotent cooperad $\CCC=(\CC,\Delta,\epsilon,\theta)$ and any element $x \in \CC(k)$, there exists a finite dimensional sub-cooperad of $\CCC$ which contains $x$.
\end{cor}

\begin{proof}
 Let $\DDD=(\DD,\Delta_\DDD, \epsilon)$ be a finite dimensional sub-graded cooperad of $\CCC$ which contains $x$. Then $\DD+ d\DD$ is a finite dimensional sub-curved cooperad of $\CCC$ which contains $x$.
\end{proof}

\begin{lemma}\label{lemma:compact}
 Any finite dimensional curved conilpotent cooperad is a compact object in the category $\cCoop$.
\end{lemma}

\begin{proof}
 The arguments of \cite[Lemma 6]{LeGrignou16} apply mutatis mutandis.
\end{proof}

\begin{proof}[Proof of Proposition \ref{prescoop}]
By Corollary \ref{cor:pres} and Lemma \ref{lemma:coopcolim}, any curved conilpotent cooperad is the colimit of the filtered diagram of its finite dimensional sub-curved conilpotent cooperads which are compact objects by Lemma \ref{lemma:compact}. Moreover, the subcategory of $\cCoop$ of finite dimensional objects is equivalent to a small category.
\end{proof}

\subsection{Product of two coaugmented cooperads}

Our goal in this section is to describe the product of two conilpotent coooperads (graded or dg).

\begin{thm}\label{prop:productcoop}
 Let $\CCC=(\CC,\Delta,\epsilon,1)$ and $\DDD=(\DD, \Delta',\epsilon',1')$ be two conilpotent cooperads (graded or differential graded). Then, the alternated tree module $\Tfree (\ov\CC, \ov\DD)$ has the structure of a conilpotent cooperad that gives us the product $\CCC \times \DDD$ in the category of conilpotent cooperads.
 \end{thm}

Let us introduce some notations
\begin{itemize}
 \itemt $\overline{\Tfree}(\ov\CC, \ov\DD)$ will denote the reduced tree module, that is the sub-$\mbs$-module of $\Tree (\ov\CC, \ov\DD)$ made up of non trivial labelled trees;
 \itemt $\overline{\Tfree}(\ov\CC, \ov\DD)_\DD$ will denote the sub-$\mbs$-module of $\Tree (\ov\CC, \ov\DD)$ made up of non trivial labelled trees whose root vertex is labelled by $\DD$;
 \itemt $\overline{\Tfree}(\ov\CC, \ov\DD)_\CC$ will denote the sub-$\mbs$-module of $\Tree (\ov\CC, \ov\DD)$ made up of non trivial labelled trees whose root vertex is labelled by $\CC$;
 \itemt finally, let ${\Tfree}(\ov\CC, \ov\DD)_\DD= \overline{\Tfree}(\ov\CC, \ov\DD)_\DD \oplus \II$ and let ${\Tfree}(\ov\CC, \ov\DD)_\CC= \overline{\Tfree}(\ov\CC, \ov\DD)_\CC \oplus \II$.
\end{itemize}

Let us describe the expected structure of a conilpotent cooperad $(\Delta_{alt},\epsilon_{alt},1_{alt})$ on $\Tfree (\ov\CC, \ov\DD)$. We describe the map $\Delta_{alt}$ by induction on the height of the trees.
\begin{itemize}
 \itemt First, on $\CC$, the map $\Delta_{alt}$ is just given by $\Delta_\CC$ and on $\DD$, this is just $\Delta_{\DD}$. We can notice that the two maps agree on their intersection $\II$.
 \itemt Let us suppose that $\Delta_{alt}$ has been defined on $\Tfree^{h\leq n} (\ov\CC, \ov\DD)$ for some $n \in \mbn$. 
 We have $\overline{\Tfree}^{h \leq n+1} (\ov\CC, \ov\DD)_\DD = \ov\DD \circ {\Tfree}^{h \leq n}(\ov\CC, \ov\DD)_\CC$.
 Then, we define $\Delta_{alt}$ on $\overline{\Tfree}^{h = n+1} (\ov\CC, \ov\DD)_\DD$ as
 \[
\begin{tikzcd}
 	 \ov\DD \circ {\Tfree}^{h \leq n}(\ov\CC, \ov\DD)_\CC
	 \arrow[d,"{\Delta \circ \Delta_{alt}}"]
	 \\
	\DD \circ \DD \circ {\Tfree}(\ov\CC, \ov\DD)_\CC^{h \leq n} \circ {\Tfree}(\ov\CC, \ov\DD)^{h \leq n}
	\arrow[d,"Id \circ Ex \circ Id "]
	\\
	\DD \circ {\Tfree}(\ov\CC, \ov\DD)^{h \leq n} \circ \DD   \circ {\Tfree}(\ov\CC, \ov\DD)^{h \leq n}
	\arrow[d,"Comp \circ Comp"]
	\\
	{\Tfree}(\ov\CC, \ov\DD)^{h \leq n+1} \circ {\Tfree}(\ov\CC, \ov\DD)^{h \leq n+1}
\end{tikzcd}
 \]
 where the exchange map $Ex: \DD \circ {\Tfree}(\ov\CC, \ov\DD)^{h \leq n} \to {\Tfree}(\ov\CC, \ov\DD)^{h \leq n}\circ \DD$ is defined as 
\begin{itemize}
 \item the canonical isomorphism $\II \circ {\Tfree}(\ov\CC, \ov\DD)^{h \leq n} = {\Tfree}(\ov\CC, \ov\DD)^{h \leq n}= {\Tfree}(\ov\CC, \ov\DD)^{h \leq n}\circ \II $
 on the subobject $\II \circ {\Tfree}(\ov\CC, \ov\DD)^{h \leq n} \subset \DD \circ {\Tfree}(\ov\CC, \ov\DD)^{h \leq n}$;
 \item the canonical isomorphism $\DD \circ \II = \DD = \II \circ \DD$ on the subobject $\II \circ \DD \subset \DD \circ {\Tfree}(\ov\CC, \ov\DD)^{h \leq n}$;
 \item the zero map on the other summands;
\end{itemize}
and where the composition map $Comp: \DD \circ {\Tfree}(\ov\CC, \ov\DD)^{h \leq n} \to {\Tfree}(\ov\CC, \ov\DD)^{h \leq n+1}$ is defined as the usual operadic composition on $\DD \circ {\Tfree}(\ov\CC, \ov\DD)_\CC^{h \leq n}$ but it is the zero map on the other summands of $\DD \circ {\Tfree}(\ov\CC, \ov\DD)^{h \leq n}$.

We define $\Delta_{alt}$ on $\overline{\Tfree}^{h = n+1} (\ov\CC, \ov\DD)_\CC$ in a similar way.
\end{itemize}
The map $\epsilon_{alt}$ is given by the projection on the trivial tree, while $1_{alt}$ is given by the inclusion of this trivial tree among all trees.

\begin{lemma}
 Suppose that $\CCC= \Tfree^c \VV$ and $\DDD= \Tfree^c\WW$. Then, we have a canonical isomorphism of $\mbs$-modules
 \[
 	\Tfree(\ov{\CC}, \ov{\DD}) = \Tfree(\VV \oplus \WW).
 \]
 Moreover, this isomorphism commutes with the decomposition structure in the sense that the following diagram commutes
 \[
\begin{tikzcd}
 	\Tfree(\ov{\CC}, \ov{\DD})
	\arrow[r,"\simeq"] \arrow[d,"\Delta_{alt}"']
	&\Tfree(\VV \oplus \WW)
	\arrow[d,"\Delta_{\Tfree^c(\VV \oplus \DD)}"]
	\\
	\Tfree(\ov{\CC}, \ov{\DD})\circ \Tfree(\ov{\CC}, \ov{\DD})
	\arrow[r,"\simeq"']
	& \Tfree(\VV \oplus \WW) \circ \Tfree(\VV \oplus \WW)
\end{tikzcd}
 \]
It also commutes with the counit and the coaugmentation. Subsequently, the maps $(\Delta_{alt},\epsilon_{alt},1_{alt})$ define the structure of a conilpotent cooperad on $\Tfree(\ov{\CC}, \ov{\DD})$ and we have an isomorphism of conilpotent cooperads
\[
	(\Tfree(\ov{\CC}, \ov{\DD}), \Delta_{alt})= \Tfree^c(\VV \oplus \WW)= \CCC \times \DDD.
\]
\end{lemma}

\begin{proof}
Any tree labelled by $\VV$ and $\WW$ has a canonical partition which consists in maximal subtrees that are only labelled by $\VV$ or that are only labelled by $\WW$. The subtrees of this partition are then elements of $\CC$ or elements of $\DD$. Moreover, since we have only taken maximal subtrees, the neighbour of an element of $\CC$ (resp. $\DD$) is necessarily an element of $\DD$ (resp. $\CC$). Thus, such a tree labelled by  $\VV$ and $\WW$ may be considered as an element of $\Tfree(\ov{\CC}, \ov{\DD})$. This gives us the isomorphism relating the $\mbs$-module $\Tfree(\ov{\CC}, \ov{\DD})$ to the $\mbs$-module $\Tfree(\VV \oplus \WW)$. The fact that it commutes with the cooperad structure maps follows from an induction on the height of trees.
\end{proof}

\begin{proof}[Proof of Theorem \ref{prop:productcoop}]
The morphism of $\mbs$-modules $\Tfree(\ov{\CC}, \ov{\DD}) \to \Tfree(\ov\Tfree\ov{\CC}, \ov\Tfree\ov{\DD})$ commutes with the maps $(\Delta_{alt},\epsilon_{alt},1_{alt})$. Thus, $(\Tfree(\ov{\CC}, \ov{\DD}, \Delta_{alt},\epsilon_{alt},1_{alt})$ is a sub-conilpotent cooperad of
\[
(\Tfree(\ov\Tfree\ov{\CC}, \ov\Tfree\ov{\DD}),\Delta_{alt},\epsilon_{alt},1_{alt} )= \Tfree^c(\ov\CC \oplus \ov\DD).
\]

Besides, let $\EEE$ be a conilpotent cooperad. Let us consider two morphisms $f: \EEE \to \CCC$ and $g: \EEE \to \DDD$. We can build a morphism of $\mbs$-modules denoted $(f,g)$ as follows
\[
	\EE \xrightarrow{\delta_\EEE} \Tfree^c \ov\EE \to \II \oplus \ov\Tfree(\ov\CC, \ov\DD)_\CC \oplus \ov\Tfree(\ov\CC, \ov\DD)_\DD
	= \Tfree(\ov\CC, \ov\DD)
\]
where the last map consists in
\begin{itemize}
 \itemt sending the trivial tree to the trivial tree;
 \itemt sending a nontrivial tree labelled by $\ov\EE$ to the same tree labelled whose root vertex is labelled by $\ov\CC$ (using $f$), the vertices just above are labelled by $\ov\DD$ (using $g$), \ldots
  \itemt sending a nontrivial tree labelled by $\ov\EE$ to the same tree labelled whose root vertex is labelled by $\ov\DD$ (using $g$), the vertices just above are labelled by $\ov\CC$ (using $f$), \ldots
\end{itemize}
One can check that the composite map
\[
	\EE \to \Tfree(\ov\CC, \ov\DD) \to \Tfree(\ov\Tfree\ov{\CC}, \ov\Tfree\ov{\DD})= \Tfree(\ov\CC \oplus \ov\DD) 
\]
is equal to the map of cooperads adjoint to the morphism of $\mbs$-modules $f+g: \ov\EE \to \ov\CC \oplus \ov\DD$. Since the map $\Tfree(\ov\CC, \ov\DD) \to \Tfree(\ov\Tfree\ov{\CC}, \ov\Tfree\ov{\DD})$ has a left inverse in the category of $\mbs$-modules and is a morphism of cooperads, then the first map $(f,g):\EE \to \Tfree(\ov\CC, \ov\DD)$ is also a morphism of cooperads. Moreover, we can recover $f$ and $g$ as the composite maps
\begin{align*}
 & \EE \xrightarrow{(f,g)}\Tfree(\ov\CC, \ov\DD) \to \Tfree(\ov\CC, \ov\II) = \CC;
 \\
 &  \EE \xrightarrow{(f,g)}\Tfree(\ov\CC, \ov\DD) \to \Tfree(\ov\II, \ov\DD) = \DD.
\end{align*}
Suppose that there we have morphism of cooperads $h: \EE \to \Tfree(\ov\CC, \ov\DD)$ that factorises $f$ and $g$ as just above. Then the two composite maps
\begin{align*}
 & \EE \xrightarrow{(f,g)}\Tfree(\ov\CC, \ov\DD) \to \Tfree^c(\ov\CC\oplus \ov\DD) ;
 \\
 &  \EE \xrightarrow{h}\Tfree(\ov\CC, \ov\DD) \to \Tfree^c(\ov\CC\oplus \ov\DD) ;
\end{align*}
are the same. Since the map $\Tfree(\ov\CC, \ov\DD) \to \Tfree^c(\ov\CC\oplus \ov\DD)$ is injective, then $h=(f,g)$. This shows that $(\Tfree(\ov{\CC}, \ov{\DD}), \Delta_{alt},\epsilon_{alt},1_{alt})= \CCC \times \DDD$.
\end{proof}

\subsection{Bar-cobar adjunction}

In this section, we recall the bar-cobar adjunction introduced in \cite{LeGrignou16} and which relates operads with curved conilpotent cooperads. Let $\PPP:=(\PP, \gamma, 1)$ be an operad. Its bar construction is the curved conilpotent cooperad $B_c \PPP := \Tfree^c (s\PP \oplus \mbk \cdot v)$,
where $v$ is a degree $2$ element. It is equipped with the coderivation which extends the following map
\begin{align*}
 \Tfree (s\PP \oplus \mbk \cdot v) \twoheadrightarrow \ov \Tfree^{\leq 2} (s\PP \oplus \mbk \cdot v) & \ra s\PP \\
 sx \otimes sy &\mapsto (-1)^{|x|} s \gamma_\PPP (x \otimes y)\\
 sx \otimes v & \mapsto 0\\
 v & \mapsto s1_\PPP\\
 sx & \mapsto -sdx\ .
\end{align*}
Its curvature is the degree $-2$ map.
\begin{align*}
 \Tfree (s\PP \oplus \mbk \cdot v) \twoheadrightarrow  \mbk \cdot v & \ra \mbk\\
 v &\mapsto 1\ .
 \end{align*}

Let $\CCC:=(\CC,\Delta, \epsilon, 1 ,\theta)$ be a curved conilpotent cooperad. Its cobar construction is the operad
$$
\Omega_u \CCC := \Tfree (s^{-1} \ov \CC)\ .
$$
It is equipped with the following derivation,
$$
s^{-1}x \mapsto  \theta (x) 1 - s^{-1} dx - \sum (-1)^{|x_1|} s^{-1}x_1 \otimes s^{-1}x_2\ ,
$$
where $\sum x_1 \otimes x_2 = \Delta_2 x$.\\

A twisting morphism is a degree $-1$ map $\alpha$ from a curved conilpotent cooperad $\CCC$ to an operad $\PPP$ such that $\alpha (1)=0$ and such that
$$
\partial (\alpha) + \gamma_\PPP (\alpha \otimes \alpha ) \Delta_\CCC = \theta_\CCC (-) 1_\PPP\ .
$$
We denote by $\Tw (\CCC, \PPP)$ the set of twisting morphisms from $\CCC$ to $\PPP$.

\begin{prop}\cite{LeGrignou16}
The bar construction and the cobar construction are both functorial. Moreover, there exist functorial isomorphisms:
 $$
 \hom_{\Operad} (\Omega_u \CCC, \PPP) \simeq \Tw (\CCC, \PPP) \simeq \hom_\cCoop (\CCC, B_c \PPP)\ .
 $$
 Therefore, the functor $\Omega_u$ is left adjoint to the functor $B_c$.
\end{prop}

\subsection{Truncated bar construction}

In this section, we recall the operadic bar construction of Hirsh-Mill\`es which we call the truncated bar construction; see \cite{HirshMilles12} for the original reference. Let $\PPP:=(\PP,\gamma, 1)$ be a dg-operad. Suppose that $\PPP$ is equipped with a semi-augmentation, that is, a morphism of graded $\mbs$-modules $\epsilon: \PP \ra \II$ such that $\epsilon(1) =Id$. We denote by $\ov\PP$ the kernel of $\epsilon$ and we denote by $\pi$ the projection of $\PP$ onto $\ov \PP$ along the unit $1$. The truncated bar construction of $\PPP$ relative to the semi-augmentation $\epsilon$ is the cofree conilpotent cooperad $B_r \PPP := \Tfree^c s\ov \PP$. It is equipped with the coderivation which extends the map
\begin{align*}
\ov \Tfree (s \ov\PP) \hookrightarrow \ov \Tfree^{\leq 2} (s \ov\PP)  &\ra s\ov \PP \\
sx \otimes sy &\mapsto (-1)^{|x|}s\pi \gamma(x \otimes y) \\
sx &\mapsto -s \pi dx\ .
\end{align*}
The curvature $\theta$ is the following map:
\begin{align*}
 \Tfree (s \ov\PP) \hookrightarrow \ov \Tfree^{\leq 2} (s \ov\PP)  &\ra \mbk \\
 sx \otimes sy &\mapsto (-1)^{|x|+1} \epsilon \gamma (x \otimes y)\\
 sx & \mapsto \epsilon (dx)\ .
\end{align*}
A truncated twisting morphism from a curved conilpotent cooperad $\CCC$ to a semi-augmented operad $\PPP$ is a twisting morphism $\alpha : \ov \CC \ra \PP$ such that $\epsilon \alpha= 0$. We denote by $trTw(\CCC,\PPP)$ the set of twisting morphisms from a curved conilpotent cooperad $\CCC$ to a semi-augmented dg-operad $\PPP$. 

\begin{prop}[\cite{HirshMilles12}]
 For any semi-augmented operad $\PPP$ and any curved conilpotent cooperad $\CCC$, we have functorial isomorphisms:
 $$
 \hom_{\mathsf{sa}-\Operad}(\Omega_u \CCC, \PPP) \simeq trTw(\CCC,\PPP) \simeq \hom_{\cCoop}(\CCC, B_r \PPP)\ ,
 $$
 where $\mathsf{sa}-\Operad$ is the category of semi-augmented dg operads.
\end{prop}

For any semi-augmented dg-operad, the universal truncated twisting morphism
$$
B_r \PPP = \Tfree s\ov \PP \twoheadrightarrow s \ov \PP \xrightarrow{s^{-1}} \ov \PP \hookrightarrow \PP
$$ 
is in particular a twisting morphism. So it induces a morphism of curved conilpotent cooperad from $B_r \PPP$ to $B_c \PPP$.


\section{Model structure on operads}

This section deals with the homotopy theory of dg operads. We recall the result proved by Hinich in \cite{Hinich97} that there exists a model structure on the category of dg operads whose weak equivalences are quasi-isomorphisms and whose fibrations are surjections. Then, we describe the cofibrations in a convenient  way to be able to use it in the sequel.

\begin{nota}
 We denote by $S^k_n$ (resp. $D^k_n$) the dg-$\mbs$-module such that 
\begin{itemize}
 \itemt in arity $n$, $S^k_n (n) = S^k \otimes \mbk[\mbs_n]$ and $D^k_n (n)= D^k \otimes \mbk[\mbs_n]$,
 \itemt if $m\neq n$, then $S^k_n(m) =D^k_n(m)=0$.
\end{itemize}
\end{nota}

\subsection{Model structure on $\mbs$-modules} 

\begin{thm}\cite[\S 2.3]{Hovey99}
For any integer $n \in \mbn$, there exists a cofibrantly generated model structure on the category of chain complexes of $\mbk[\mbs_n]$-modules whose fibrations (resp. weak equivalences) are epimorphisms (resp. quasi-isomorphisms). Moreover, the cofibrations are the morphisms whose cokernel is cofibrant and which are degreewise split monomorphisms.
\end{thm}

\begin{prop}
Since the characteristic of $\mbk$ is zero, any degreewise monomorphism of chain complexes of $\mbk[\mbs_n]$-modules is a cofibration.\\ 
\end{prop}

\begin{proof}
 By Maschke's Theorem, any monomorphism of $\mbs_n$-modules is split. So we just have to show that any chain complex is cofibrant. It suffices to show that any acyclic fibration splits.
\end{proof}

Subsequently, there exists a cofibrantly generated model structure on the category of dg $\mbs$-modules whose cofibrations (resp. fibrations, resp. weak equivalences) are exactly monomorphims (resp. epimorphisms, resp. quasi-isomorphisms). Moreover, a set of generating cofibrations is made up of the maps $0 \ra S^0_n$ and $S^k_n \ra D^{k+1}_n$; a set of generating acyclic cofibrations is made up of the maps $0 \ra D^k_n$.

\subsection{Model structure on operads}

Consider the adjunction
\[
\begin{tikzcd}
 \text{dg }\mbs-\text{mod} \arrow[r, shift left, "\Tfree"] & \Operad\ , \arrow[l, shift left, "U"] 
\end{tikzcd}
\]
where $U$ is the forgetful functor. Transferring this model structure along this adjunction gives the following Theorem.
 
\begin{thm}[\cite{Hinich97}]\label{thm:hinich}
The category of dg operads admits a cofibrantly generated model structure where the weak equivalences (resp. fibrations) are the componentwise quasi-isomorphisms (resp. epimorphisms). The generating cofibrations are the maps $\II \ra \Tfree (S^0_n)$ and the maps $\Tfree (S^k_n) \ra \Tfree (D^{k+1}_n)$. The generating acyclic cofibrations are the maps $\II \ra \Tfree (D^k_n)$.
\end{thm}

\subsection{Cofibrations of operads}

We prove the following proposition in the vein of \cite[Appendix 1]{MerkulovVallette09I}.

\begin{prop}\label{prop:cofibop}
 Cofibrations of operads are exactly retracts of morphisms $\PPP \ra \PPP \vee \Tfree S $ where $S$ is a $\mbs$-module endowed with an exhaustive filtration
$$
 \{0\} = S_0 \subset S_1 \subset ... \subset \colim_{\beta < \alpha} S_\alpha\ ,
 $$
 indexed by an ordinal $\alpha$ and such that
$$
d(S_{i+1}) \subset S_{i+1} \oplus \PPP \vee  \Tfree(S_{i})\ .
$$
\end{prop}

\begin{proof}
Given the generating cofibrations of the category of operads given in Theorem \ref{thm:hinich} and by \cite{Hovey99}, any cofibration of operads is a retract of a morphism $\PPP \ra \PPP \vee \Tfree S$ as in Proposition \ref{prop:cofibop} with the additional conditions that the cokernel of the inclusions $S_i \ra S_{i+1}$ are free $\mbs$-modules (that is the cokernel is a free $\mbk[\mbs_n]$-module in arity $n$) and that
$$
d(S_{i+1}) \subset \PPP \vee  \Tfree(S_{i})\ .
$$
Conversely, consider a morphism of operads $f: \PPP \ra \PPP \vee \Tfree S$ such that $d( S)\subset S \oplus \PPP$. It fills the following pushout diagram
$$
\xymatrix{\Tfree ( s^{-1}S) \ar[d] \ar[r] & \PPP \ar[d]\\
 \Tfree (S \oplus s^{-1}S) \ar[r] & \PPP \vee \Tfree S}
$$
where $S \oplus s^{-1}S$ is endowed with the differential $x +s^{-1} y \mapsto dx + s^{-1} x - s^{-1}dy$. Since the map $ s^{-1}S \ra S \oplus s^{-1}S$ is a cofibration of dg $\mbs$-modules (since it is a monomorphism), then $\Tfree (s^{-1}S) \ra \Tfree(S \oplus s^{-1}S)$ is a cofibration of operads and so $f$ is a cofibration. Any morphism $\PPP \ra \PPP \vee \Tfree S$ as in Proposition \ref{prop:cofibop} is a transfinite composite of morphisms as $f$ and so is a cofibration.
\end{proof}

\subsection{Enrichment in simplicial sets}

Let $\PPP=(\PP,\gamma_\PPP,1_\PPP)$ be an operad and let $\AAA:=(\Aa,\gamma_\AAA,1_\AAA)$ be a unital commutative algebra. Let $\PP \otimes \Aa$ be the $\mbs$-module defined by
$$
(\PP \otimes \Aa) (m):= \PP(m) \otimes \Aa\ .
$$
It has an obvious structure of operad. This construction is functorial.\\

Besides for any integer $n \in \mbn$, let $\Omega_n$ be the unital commutative algebra
$$
\Omega_n:= \mbk [t_0, \ldots, t_n ,dt_0, \ldots, dt_n]/(\Sigma t_i=1; \Sigma dt_i = 0)\ .
$$
The construction $n \mapsto \Omega_n$ defines a  simplicial unital commutative algebra. This provides an enrichment of the category of dg operads over simplicial sets as follows:

$$
HOM(\PPP,\QQQ)_n := \hom_{\Operad}(\PPP , \QQQ \otimes \Omega_n) = \hom_{\Operad_{\Omega_n}}(\PPP \otimes \Omega_n , \QQQ \otimes \Omega_n)\ .
$$

\begin{prop}
 For any dg operads $\PPP$ and $\QQQ$ with $\PPP$ cofibrant, the simplicial set $HOM(\PPP,\QQQ)$ is a model for the mapping space $\Map(\PPP,\QQQ)$.
\end{prop}

\begin{proof}
 It suffices to notice that the simplicial operad $(\QQQ \otimes \Omega_n)_{n \in \mbn}$ is a Reedy fibrant replacement of the constant simplicial operad $\QQQ$.
\end{proof}

\section{Model structure on curved conilpotent cooperads}\label{section:model}

In this section, we show that the model structure on the category of dg operads can be transferred through the cobar construction functor to the category of curved conilpotent cooperads. This result is in the vein of earlier results by Hinich \cite{Hinich01}, Lefevre-Hasegawa \cite{LefevreHasegawa03}, Vallette \cite{Vallette14} and Positselski \cite{Positselski11}. Our proof relies on the same kind of method; however new difficulties appear with the combinatorics of trees and actions of symmetric groups.

\subsection{Statement of the result}

Here is the main result of this paper. The remaining of this section will be its proof.

\begin{thm}\label{thm:thmprincip}
 There exists a model structure on the category of curved conilpotent cooperads whose cofibrations (resp. weak equivalences) are the morphisms whose image under the cobar construction functor $\Omega_u$ is a cofibration (resp. a weak equivalence). Moreover, the adjunction $\Omega_u\dashv B_c$ is a Quillen equivalence.
\end{thm}

From now on, we call \textit{cofibrations} (resp. \textit{weak equivalences}) of curved conilpotent cooperads the morphisms whose image under the cobar functor $\Omega_u$ is a cofibration (resp. weak equivalence). Moreover, we call \textit{acyclic cofibrations} the morphisms which are both cofibrations and weak equivalences and we call \textit{fibrations} the morphisms which have the right lifting property with respect to acyclic cofibrations. Finally, we call \textit{acyclic fibrations} the morphisms which are both fibrations and weak equivalences.

\subsection{Cofibrations}

We describe the cofibrations of curved conilpotent cooperads.

\begin{prop}
 The cofibrations of curved conilpotent cooperads are the degreewise injections.
\end{prop}

\begin{proof}
 Let $f: \CCC \ra \DDD$ be a morphism of curved conilpotent cooperads which is degreewise injective. It is the transfinite composite of the morphisms
 $$
 f_n : \DDD_{[n]} \ra \DDD_{[n+1]}\ ,
 $$
where $\DDD_{[n]} =\CCC + F_n^{rad}\DDD \subset \DDD$.  Let us decompose the underlying graded $\mbs$-module of $\DD[n+1]$ as
\[
	\DD_{[n+1]} = \DD_{[n]} \oplus E
\]
At the level of graded operads, the morphism $\Omega_u(f_n)$ is just 
\[
	\Tfree(s^{-1}\overline{\DD}_{[n]}) \to \Tfree(s^{-1}\overline{\DD}_{[n]} \oplus s^{-1} E) .
\]
Moreover, as a consequence of the property of coradical filtrations recalled in Definition \ref{definicorad}, the morphism $E \hookrightarrow \DD_{[n+1]} \xrightarrow{\Delta_2} \Tfree^2 \overline{\DD}_{[n+1]}$ has its image lying in $\Tfree^2 \overline{\DD}_{[n]}$. Hence the restriction of the derivation of $\Omega_u(\DDD_{[n+1]})$ to $s^{-1} E$ may be decomposed as follows
\begin{itemize}
 \itemt the curvature part, that targets $\mbk \subset \Omega_u(\DDD_{[n]}) \subset \Omega_u(\DDD_{[n+1]})$; let us denote it $d_{curv}$;
\itemt the part coming from the decomposition $\Delta_2$, that targets $\Omega_u(\DDD_{[n]}) \subset \Omega_u(\DDD_{[n+1]})$; let us denote it $d_{dec}$;
 \itemt the part coming from the coderivation of $\DDD$, that targets $s^{-1}\overline{DD}_{[n+1]} = s^{-1}\overline{DD}_{[n]} \oplus s^{-1}E$; let us denote it $d_{twist}$ the part that targets $s^{-1}\overline{DD}_{[n]}$ and $d_E$ the part that targets $s^{-1}E$.
 \end{itemize}
One can show that $d_E^2=0$ and makes $s^{-1}E$ a dg operad. Then, the following diagram of dg operads is a pushout
\[
\xymatrix{
\Tfree (S^{-1} \otimes s^{-1}E)
\ar[r] \ar[d]
& \Omega_u (\DDD_{[n]})
\ar[d]
\\ 
\Tfree (D^0 \otimes s^{-1}E)
\ar[r]
& \Omega_u(\DDD_{[n+1]})
}
\]
where the top horizontal map is the composition of graded maps
\[
	S^{-1} \otimes s^{-1}E \xrightarrow{|1|} s^{-1}E \xrightarrow{d_{curv}+ d_{dec}+ d_{twist}} \Omega_u(\DDD_{[n]}) .
\]
Then, $\Omega_u (f_n)$ is a morphism of operads of the form $\PPP \ra \PPP \vee \Tfree (S)$ as in Proposition \ref{prop:cofibop}.
So, $\Omega_u (f_n)$ is a cofibration. Since cofibrations of operads are stable under transfinite composition, then $\Omega_u (f)$ is a cofibration. So $f$ is a cofibration.

Conversely, if $f: \CCC \ra \DDD$ is a cofibration, then $\Omega_u (f)$ is a cofibration and in particular it is injective. Indeed, generating cofibrations and their coproducts are degreewise injective and degreewise injections are stable under transfinite composition and retract. So the composite map $s^{-1} \ov\CC \hookrightarrow \Omega_u \CCC \ra \Omega_u \DDD$ is injective. Since it factorises through the map
$s^{-1}\ov\CC \ra s^{_1}\ov \DD$, the latter is also degreewise injective. Hence, the map $\ov\CC \ra \ov \DD$ is injective and so $f$ is injective.
\end{proof}

\subsection{Weak equivalences and filtered quasi-isomorphisms}
Weak equivalences of curved conilpotent cooperads are morphisms whose image under the functor cobar $\Omega_u$ is a quasi-isomorphism. Giving their explicit description is not an easy task. A sufficient condition for a morphism to be a weak equivalence is to be a filtered quasi-isomorphism.

\begin{defin}[Admissible filtrations and filtered quasi-isomorphisms]
Let $\CCC:= (\CC,\Delta,\epsilon, 1 , \theta)$ be a curved conilpotent cooperad. An admissible filtration $(F_n \CC)_{n \in \mbn}$ of $\CCC$ is an exhaustive filtration of the $\mbs$-module $\CC$ satisfying the following conditions.

$$
\begin{cases}
 d(F_n \CC) \subset F_n \CC\ ,\\
\Delta (F_n \CC)(m) \subset \sum_{p_0 + \cdots + p_k \leq n\atop k \in \mbn} (F_{p_0} \CC)(k) \otimes_{\mbs_k} (F_{p_1} \CC \circledast \cdots \circledast F_{p_k} \CC )\ ,\\
 F_0 \CC:= \II.
\end{cases}
 $$
 Let $\CCC$ and $\DDD$ be two curved conilpotent cooperads both equipped with an admissible filtration. A filtered quasi-isomorphism from $\CCC$ to $\DDD$ relative to these two filtrations is a morphism $f:\CCC \ra \DDD$ which preserves these filtrations and such that the induced morphism
 $$
 Gf : G\CCC \ra G\DDD
 $$
 is a quasi-isomorphism.
\end{defin}

\begin{eg}
 We know from \cite[Lemma 1]{LeGrignou16} that the coradical filtration of a curved conilpotent cooperad is admissible.
\end{eg}

\begin{prop}\label{prop:filteredqis}
 A filtered quasi-isomorphism is a weak equivalence.
\end{prop}

We will use the following Theorem to prove this proposition.

\begin{thm}\cite[XI.3.4]{MacLane95}\label{maclane-homology}
 Let $f: \VV \rightarrow \WW$ be a map of filtered chain complexes. Suppose that the filtrations are bounded below and exhaustive. If for any integer $n$, the map $G_n \VV \rightarrow G_n \WW$ is a quasi-isomorphism, then $f$ is a quasi-isomorphism.
\end{thm}

\begin{proof}[Proof of Proposition \ref{prop:filteredqis}]
Let $f: \CCC \ra \DDD$ be a filtered quasi-isomorphism. Consider the filtration on $\Omega_u \CCC$ induced by the coradical filtration of $\CC$ (Definition \ref{defininducedfilt}):
$$
F_n \Omega_u \CCC (q):= \bigoplus_T  \sum_{i_1 + \cdot +i_k =n} s^{-1} F_{i_1}^{rad} \ov \CC \otimes \cdots \otimes s^{-1} F_{i_k}^{rad} \ov \CC  
$$
where $k$ is the number of vertices of the tree $T$, for $n$ varying from $0$ to $\infty$. Then, $G( \Omega_u \CCC) =\Tfree G(\ov\CC)$.
Let us endow $\Omega_u \DDD$ with a filtration built in the same fashion. For any integer $n$, let us endow $G_n \Omega_u \CCC$ with the following filtration
$$
F'_k G_n \Omega_u \CCC := \bigoplus_{p\geq -k}\ \  \sum_{i_1+ \cdots + i_p = n}  s^{-1}G_{i_1} \ov \CC \otimes \cdots \otimes s^{-1} G_{i_p} \ov \CC  
$$
for $k$ varying from $-n$ to $0$. In other words, $F'_k G_n \Omega_u \CCC$ is made up of trees with $-k$ or more vertices. Again, we endow $G_n\Omega_u \DDD$ with a filtration built in the same fashion. Then, the map
$$
G'G_nf: G'G_n\Omega_u \DDD \ra  G'G_n\Omega_u \CCC
$$
is a quasi-isomorphism. We conclude by Theorem \ref{maclane-homology}.
\end{proof}

\begin{prop}\label{prop:truncvsusual}
 Let $\PPP$ be an operad together with a semi-augmentation $\epsilon : \PP \ra \II$. Then the canonical morphism $B_r \PPP \ra B_c \PPP$ is a filtered quasi-isomorphism with respect to the coradical filtrations. Hence, it is a weak equivalence.
\end{prop}

\begin{proof}
It suffices to notice that the morphism of chain complexes
\begin{align*}
 s\ov \PP & \ra s\PP \oplus \mbk \cdot v\\
 sx &\mapsto sx + \theta (sx) v
\end{align*}
is a quasi-isomorphism.
\end{proof}

\subsection{Bar-cobar and cobar-bar resolutions}
Let $\VV$ be a dg $\mbs$-module. The tree module $\Tfree (s^{-1} \ov \Tfree \VV)$ has both a structure of operad and cooperad. Let $D$ be the derivation which makes $\Tfree (s^{-1} \ov \Tfree \VV)$ the cobar construction of the dg conilpotent cooperad $\Tfree \VV$, that is: 
 $$
 s^{-1} A \ra  -s^{-1} dA - \sum (-1)^{|A_1|}s^{-1}A_1 \otimes s^{-1 } A_2
 $$
 for any $A \in \ov \Tfree \VV$, where $\Delta_2 A = \sum A_1 \otimes A_2$. Besides, let $h$ be the degree $1$ coderivation of the cooperad $\Tfree (s^{-1} \ov \Tfree \VV)$ which extends the following map
\begin{align*}
 \Tfree (s^{-1} \ov \Tfree \VV) \twoheadrightarrow \Tfree^2 (s^{-1} \ov \Tfree \VV) & \ra  \ov \Tfree \VV\\
 s^{-1} A_1 \otimes s^{-1}A_2 & \mapsto s^{-1}(A_1 \otimes A_2)
\end{align*}

\begin{lemma}\label{lemmacontractcobar}
Let $T$ be a tree with $k$ vertices ordered from $1$ to $k$ and let $T_1$, \ldots, $T_k$ be non trivial trees. Consider the sub $\mbs$-module of  $\Tfree (s^{-1} \ov \Tfree \VV)$ made up of the tree $T$ whose $i^{th}$ vertex is labelled by $T_i (\VV)$. On this submodule, we have:
$$
Dh +h D= q Id\ ,
$$
where $q$ is the sum of the numbers of inner edges of $T$ and of the trees $T_i$.
\end{lemma}

\begin{proof}
Let $a$ be an inner edge of the tree $T$. It links two vertices which are labelled respectively by the tree module $s^{-1}T_i(\VV)$ and the tree module $s^{-1}T_j(\VV)$. The derivation $h$ consists in grafting the tree $T_i$ with the tree $T_j$ for any inner edge $a$. We can write:
$$
h(T):=\sum_{a \in inner(T)} \text{graft}(a)\ ,
$$
where $inner (T)$ is the set of inner edges of $T$. Moreover, let $x$ be a vertex of the tree $T$. It is labelled by the tree module $s^{-1}T_i(\VV)$. The derivation $D$ consists in cutting the tree module $s^{-1}T_i(\VV)$ into two trees $s^{-1}T_{i,1}(\VV) \otimes s^{-1}T_{i,2}(\VV)$ along any inner edge of $T_i$, and then applying the differential $d$ of $\VV$; and that is done for any vertex $x$ of the tree $T$. So, we can write
$$
D(T)= \sum_{a \in inner (T_1, \ldots, T_k)} \text{cut}(a) +\sum_{x \in vert (T_1, \ldots, T_k)} \text{diff}(x)\ ,
$$  
where $inner (T_1, \ldots, T_k )$ is the set of inner edges of the trees $T_1$, \ldots, $T_k$ and $vert (T_1, \ldots, T_k)$ is the set of vertices of the trees $T_1$, \ldots, $T_k$. So, we have: 
 $$
  hD =   \sum_{a \in inner (T_1, \ldots, T_k) \atop b \in inner (T)} \text{graft}(b)\text{cut}(a)+ \sum_{a \in inner (T_1, \ldots, T_k)} \text{graft}(a)\text{cut}(a) +  \sum_{x \in vert (T_1, \ldots, T_k) \atop a \in inner (T)} \text{graft}(a) \text{diff}(x)\ .
  $$
On the other hand,
  $$
    Dh =  \sum_{a \in inner (T)} \text{cut} (a) \text{graft} (a) + \sum_{a \in inner (T_1, \ldots, T_k) \atop b \in inner (T)} \text{cut}(a)\text{graft}(b) +  \sum_{x \in vert (T_1, \ldots, T_k) \atop a \in inner (T)} \text{diff}(x)\text{graft}(a)\ .
  $$
  For any inner edge $a$ of the trees $T_1$, \ldots, $T_k$ and for any inner edge $b$ of the tree $T$, $\text{cut}(a)\text{graft}(b) +\text{graft}(b)\text{cut}(a)=0$. Moreover, for any inner edge $a$ of the tree $T$ and for any vertex $x$ of the trees $T_1$, \ldots, $T_k$, $\text{diff}(x)\text{graft}(a) + \text{diff}(x)\text{graft}(a) =0$. Finally, for any inner edge $a$ of the trees $T_1$, \ldots, $T_k$ and $T$, $\text{cut} (a) \text{graft} (a) + \text{graft} (a) \text{cut} (a) =Id$.
\end{proof}

Similarly, the $\mbs$-module $\Tfree (s \ov \Tfree \VV)$ has a structure of operad and a structure of cooperad. Let $D$ be the coderivation which makes of $\Tfree (s (\ov\Tfree \VV) )$ the truncated bar construction of the dg operad $\Tfree \VV$; that is, the projection of $D$ on the cogenerators is defined as follows
\begin{align*}
sA_1 \otimes sA_2 &\mapsto (-1)^{|A_1|}s(A_1 \otimes A_2)\\
 sA &\mapsto -sdA\\ 
\end{align*}
for any $A \in \ov \Tfree \VV$. Moreover, let $h$ be the degree $1$ derivation which extends the following map.
\begin{align*}
 s \ov\Tfree \VV & \ra \Tfree (s \ov \Tfree \VV)\\
 sA &\mapsto \sum (-1)^{|A_1|}sA_1 \otimes sA_2
\end{align*}
where $\Delta_2 A = \sum A_1 \otimes A_2$ for any $A \in \ov \Tfree \VV$.

\begin{lemma}\label{lemmacontractbar}
Let $T$ be a tree with $k$ vertices ordered from $1$ to $k$ and let $T_1$, \ldots, $T_k$ be non trivial trees. Consider the sub $\mbs$-module of  $\Tfree (s \ov \Tfree \VV)$ made up of the tree $T$ whose $i^{th}$ vertex is labelled by $T_i (\VV)$. On this submodule, we have:
$$
Dh +h D= q Id\ ,
$$
where $q$ is the sum of the numbers of inner edges of $T$ and of all the trees $T_i$.
\end{lemma}

\begin{proof}
 The proof relies on the same techniques as Lemma \ref{lemmacontractcobar}.
\end{proof}

\begin{prop}\label{prop:eq1}
  Let $\PPP$ be a dg operad. Then the canonical morphism $p: \Omega_u B_c \PPP \ra \PPP$ is a weak equivalence.
\end{prop}

\begin{proof}
Consider the filtration on $\Omega_u B_c \PPP$ induced by the coradical filtration on $B_c \PPP$:
 $$
 F_n \Omega_u B_c \PPP:= \II \oplus \bigoplus_{k \geq 1} \sum_{i_1 + \cdots + i_k =n } s^{-1}F_{i_1}^{rad} \ov B_c \PPP \otimes \cdots \otimes s^{-1}F_{i_1}^{rad} \ov B_c \PPP\ ,\ n \geq 1\ . 
 $$
Consider also the constant filtration $(F_n \PPP)_{n \geq 1}$ on $\PPP$. We show that the morphisms $G_n p: G_n \Omega_u B_c \PPP \ra G_n \PPP$ are quasi-isomorphisms. On the one hand, $G_1 p: \mbk \cdot 1 \oplus \mbk \cdot s^{-1}v \oplus s^{-1}s\PPP \ra \PPP$ is a quasi-isomorphism. On the other hand, for any $n>1$, $G_n \PPP = 0$ and $G_n \Omega_u B_c \PPP$ is contractible by Lemma \ref{lemmacontractcobar}. We conclude by Theorem \ref{maclane-homology}.
\end{proof}

A straightforward consequence of the above Proposition \ref{prop:eq1} is that for any curved conilpotent cooperad $\CCC$, the map $\CCC \ra B_c \Omega_u \CCC$ is a weak equivalence. Indeed, since the morphism $\Omega_u B_c \Omega_u \CCC \ra \Omega_u \CCC$ is a quasi-isomorphism, then  its right inverse $\Omega_u \CCC \ra \Omega_u B_c \Omega_u \CCC$ is also a quasi-isomorphism. Moreover, since $B_r \Omega_u \CCC \to B_c \Omega_u \CCC$ is a weak equivalence by Proposition \ref{prop:truncvsusual}, then the morphism $\CCC \to B_r \Omega_u \CCC$ is also a weak equivalence. The following proposition is a more precise statement.

\begin{prop}\label{prop:bar-cobarres}
 Let $\CCC$ be a curved conilpotent cooperad. Let us endow $\CCC$ with its coradical filtration and let us endow $B_r \Omega_u \CCC$ with the following filtration:
 $$
 F_n B_r \Omega_u \CCC := \II \oplus \sum_{i_1 + \cdot +i_k = n\atop k \geq 1} s F_{i_1} \ov \Omega_u  \CCC  \otimes \cdots \otimes s F_{i_k} \ov \Omega_u \CCC  \ ,\ n \geq 0\ ,
 $$
 where the $(F_i \Omega_u  \CCC)_i$ is the filtration on $\Omega_u \CCC$ induced by the coradical filtration on $\CCC$
 $$
 F_{i} \ov \Omega_u \CCC :=  \sum_{j_1 + \cdot +j_k = i\atop k \geq 1}s^{-1}F_{j_1}^{rad} \ov \CC \otimes \cdots \otimes s^{-1}F_{j_k}^{rad} \ov \CC , \ i \geq 1\ .
 $$
 These two filtrations are admissible and the canonical morphism $\CCC \ra B_r \Omega_u \CCC$ is a filtered quasi-isomorphism with respect to these filtrations.
\end{prop}

\begin{rmk}
 Beware! We use here the truncated bar construction. The only reason is that $\Omega_u \CCC$ has a canonical semi-augmentation and that computations are easier with the truncated bar construction since it is smalller.
\end{rmk}

\begin{proof}
 Let $n \geq 1$. Let us show that the morphism $G_n \CCC \ra G_n B_r \Omega_u \CCC$ is a quasi-isomorphism. Consider the filtration $(F'_{k}G_n B_r \Omega_u \CCC)_{k=-n}^{-1}$ on $G_n B_c \Omega_u \CCC$ where $F'_{k}G_n B_c \Omega_u \CCC$ is made up of the trees whose vertices are labeled by trees whose total number of vertices is at least $-k$. Consider also the filtration $(F'_k G_n^{rad} \CCC)_{k=-n}^{-1}$ of $G_n^{rad} \CCC$ such that $F'_k G_n^{rad} \CCC=0$ for $k<-1$ and $F'_{-1} G_n^{rad} \CCC= G_n^{rad} \CCC$. The map $G'_{-1}G_n^{rad} \CCC \ra G'_{-1}G_n B_r \Omega_u \CCC$ is a quasi-isomorphism; that is the identity of $G_n^{rad} \CCC$. Moreover, $G'_k G_n B_r \Omega_u \CCC$ for $k \neq -1$ is contractible by Lemma \ref{lemmacontractbar}.
\end{proof}

\subsection{Key lemma}

\begin{lemma}[Key Lemma]\label{thm:keylemma}
 Let $\CCC$ be a curved conilpotent cooperad and let $p:\PPP \ra \Omega_u \CCC$ be a fibration of operads (that is a surjection). Consider the following square:
 $$
 \xymatrix{B_c \PPP \ar[r]^{B_c p} & B_c \Omega_u \CCC \\ 
 \DDD \ar[u]  \ar[r] & \CCC \ar[u]}
 $$
 where $\DDD$ is the pullback $B_c \PPP \times_{B_c \Omega_u \CCC} \CCC$. Then, the morphism $\DDD \ra B_c \PPP$ is a weak equivalence.
\end{lemma}

The newt lemma shows that we can actually use the truncated bar construction in the lemma above.

\begin{lemma}\label{lemmatrc}
The square of Lemma \ref{thm:keylemma} factorises into two pullback squares
\[
\begin{tikzcd}
 	B_c \PPP
	\arrow[r,"{B_c p}"]
	& B_c \Omega_u \CCC
	\\
	B_r \PPP
	\arrow[r,"{B_r p}"] \arrow[u]
	& B_r \Omega_u \CCC 
	\arrow[u]
	\\
	 \DDD .\ar[u]  \ar[r] & \CCC \ar[u]
\end{tikzcd}
\]
Moreover, their image in the category of graded conilpotent cooperad are also pullbacks.
\end{lemma}

\begin{proof}
A morphism of curved conilpotent cooperads from $\EEE$ to $B_c \PPP \times_{B_c \Omega_u \CCC} B_r\Omega_u \CCC$ is equivalent to the data of a twisting morphism from $\EEE$ to $\PPP$ so that the composite twisting morphism $\EEE \to \PPP \to \Omega_u \CCC$ is truncated. The latter is truncated if and only if the former is truncated. Thus, such a morphism from $\EEE$ to $B_c \PPP \times_{B_c \Omega_u \CCC} B_r\Omega_u \CCC$ is equivalent to the data of a truncated twisting morphism from $\EEE$ to $\PPP$. In other words, the square
\[
\begin{tikzcd}
 	B_c \PPP
	\arrow[r,"{B_c p}"]
	& B_c \Omega_u \CCC
	\\
	B_r \PPP
	\arrow[r,"{B_r p}"] \arrow[u]
	& B_r \Omega_u \CCC 
	\arrow[u]
\end{tikzcd}
\]
is a pullback. Moreover, the morphism $\CCC \to B_c \Omega_u \CCC$ factorises as
\[
	\CCC \to B_r \Omega_u \CCC \to B_c \Omega_u \CCC.
\]
since the twisting morphism $\CCC \to \Omega_u \CCC$ is truncated. Thus, the square of Lemma \ref{thm:keylemma} factorises into two squares. Since the big square is a pullback and since the upper small square is a pullback, then, the lower small square is a pullback. 

Let us show that the lower square is a pullback in the category of graded conilpotent cooperads. Such a pullback is the biggest sub-graded-cooperad $\EEE$ of $B_r \PPP$ whose image under $B_r(p)$ lies inside $\CCC$. One can check that $\EEE + d\EEE \subset B_r \PPP$ is a sub-cooperad whose image under $B_r(p)$ also lies in $\CCC$. Thus $\EEE + d\EEE=\EEE$. In other words, $\EEE$ is stable under the coderivation of $B_r \PPP$. So this is a sub curved cooperad of $B_r\PPP$. Then, a morphism of curved conilpotent cooperad targeting $\EEE$ is equivalent to a morphism targeting $B_r\PPP$ whose image under $B_r(p)$ lies inside $\CCC$. This means that $\EEE= B_r\PPP \times_{B_r\Omega_u\CCC} \CCC$.

The fact that the upper square is a pullback in the category of graded conilpotent cooperads follows from the same arguments.
\end{proof}

\begin{proof}[Proof of Lemma \ref{thm:keylemma}]
By, Lemma \ref{lemmatrc} and Proposition \ref{prop:truncvsusual}, it suffices to prove that the morphism $\DDD \to B_r\PPP$ is an equivalence. By Maschke's Theorem, there exists a map of graded $\mbs$-modules $\tilde i :\Omega_u \CCC \ra \PP$ such that $p \tilde i =Id_{\Omega_u \CCC}$.
 The restriction of $\tilde i$ to $s^{-1}\ov\CC$ extends to a morphism of graded operads $i: \Omega_u \CCC \ra \PPP$. We again have $pi=Id$. Let $K$ be the kernel of $p$. We have the following isomorphism of graded cooperads:
$$
\DDD \simeq \CCC \times \Tfree^c (sK) = \Tfree(\overline \CC, \overline \Tfree (sK))
$$
Let us endow $B_r \PPP$ with the following filtration
$$
F_n B_r \PPP :=\II \oplus  \sum_{i_1 + \cdots + i_k =n\atop  k \geq 1} sF_{i_1}\ov \PP  \otimes \cdots \otimes sF_{i_k} \ov\PP , \ n \geq 0\ ,
$$
where
$$
F_i \ov  \PP :=  K \oplus  \sum_{j_1 + \cdots + j_k=i\atop k \geq 1}  s^{-1}F_{j_1}^{rad} \ov \CC \otimes \cdots \otimes s^{-1}F_{j_k}^{rad} \ov \CC  \ ,\ i \geq 1
$$
This induces a filtration on $\DDD$. These two filtrations are admissible. Let us show that the morphism $i : \DDD \ra B_r \PPP$ is a filtered quasi-isomorphism. The dg $\mbs$-module $G_n \DDD$ (resp. $G_n B_r \PPP$) is made up of trees whose vertices are labelled by $sK$ and $\CCC$ (resp. $B_r \Omega_u \CCC$). If we denote by $F'_k G_n \DDD$ the sub dg $\mbs$-module of $G_n \DDD$ made up of trees such that at least $-k$ vertices are labelled by $sK$, we obtain a bounded below filtration on $G_n \DDD$; moreover, we define the filtration $F' G_n B_r \PPP$ in the same fashion. The map
$$
G' G_n \DDD \ra G' G_n B_r \PPP
$$
is a quasi-isomorphism by Proposition \ref{prop:bar-cobarres}. We conclude by Theorem \ref{maclane-homology} and Proposition \ref{prop:filteredqis}.
\end{proof}

\subsection{Proof of Theorem \ref{thm:thmprincip}}

We gather the results proven above to prove Theorem \ref{thm:thmprincip}. We use the same steps as the proof of Theorem 3.1 in \cite{Hinich01}.

\begin{proof}[Proof of Theorem \ref{thm:thmprincip}]\leavevmode
\begin{itemize}
 \itemt The category of curved conilpotent cooperads is presentable. So, it is complete and cocomplete.
 \itemt Let $f$ and $g$ be two composable morphisms of curved conilpotent cooperads. It is clear that $f$, $g$ and $fg$ are all weak equivalences if two of them are weak equivalences since it is the case for $\Omega_u f$, $\Omega_u g$ and $\Omega_u fg$.
 \itemt Cofibrations and weak equivalences are stable under retracts because it is the case for cofibrations and weak equivalences of operads. Since they are the morphisms which satisfy the right lifting property with respect to acyclic cofibrations, the fibrations are also stable under retracts.
 \itemt Let $f: \CCC \ra \DDD$ be a morphism of curved conilpotent cooperads. Let us factorise the morphism of operads $\Omega_u (f)$ by a cofibration followed by an acyclic fibration $\Omega_u \CCC \ra \PPP \ra \Omega_u \DDD$ (resp. an acyclic cofibration followed by a fibration). Let us consider the following diagram.
 $$
 \xymatrix{B_c \Omega_u \CCC \ar[r] & B_c \PPP \ar[r] &B_c \Omega_u \DDD\\
 \CCC \ar[u] \ar[r] & B_c \PPP \times_{B_c \Omega_u \DDD} \DDD \ar[u]\ar[r] & \DDD.\ar[u]}
 $$
 The map $B_c \PPP \to B_c \Omega_u \DDD$ is a fibration (as the image under $B_c$ of a fibration). Thus the map $B_c \PPP \times_{B_c \Omega_u \DDD} \DDD \to \DDD$ is also a fibration. Moreover the map $\CCC \to B_c \PPP \times_{B_c \Omega_u \DDD} \DDD$ is a cofibration (since the composite map
$\CCC \to B_c \PPP$ is injective). Moreover, all the vertical maps are equivalences by Lemma \ref{thm:keylemma} and Proposition \ref{prop:bar-cobarres}.
Suppose that the morphism $\Omega_u \CCC\to \PPP$ is an acyclic cofibration. Then, the morphism $B_c\Omega_u \CCC\to B_c\PPP$ is an equivalence and by the 2-out-of-3 rule, the map $\CCC \to B_c \PPP \times_{B_c \Omega_u \DDD} \DDD$ is an equivalence. Suppose that the morphism $\PPP \to\Omega_u \DDD$ is an acyclic fibration. Then, the morphism $B_c\PPP\to B_c\Omega_u \DDD$ is an equivalence and by the 2-out-of-3 rule, the map $B_c \PPP \times_{B_c \Omega_u \DDD} \DDD \to \DDD$ is an equivalence.
 \itemt Consider the following square of curved conilpotent cooperads,
 $$
 \xymatrix{\CCC \ar[d]_f \ar[r] & \EEE \ar[d]^g\\
 \DDD \ar[r] & \FFF}
 $$
 where $f$ is a cofibration and $g$ is an acyclic fibration. By Lemma \ref{thm:keylemma}, $g$ can be factorised as follows
 $$
 \xymatrix{\EEE \ar[r]^(0.3){g_1} & B_c \PPP \times_{B_c \Omega_u \FFF} \FFF \ar[r]^(0.7){g_2} & \FFF}
 $$
 where $g_1$ is an acyclic cofibration and where $g_2$ is the pullback of a map $B_c \PPP \ra B_c \Omega_u \FFF$ which is the image under the functor $B_c$ of an acyclic fibration of operads $\PPP \ra \Omega_u \FFF$. Since $\Omega_u(f)$ has the left lifting property with respect to this map $\PPP \ra \Omega_u \FFF$, then $f$ has the left lifting property with respect to $g_2$. Moreover, the following square has a lifting by definition of the fibrations.
 $$
 \xymatrix{\EEE \ar[d]_{g_1} \ar[r]^= & \EEE \ar[d]^{g}\\
 B_c \PPP \times_{B_c \Omega_u \FFF} \FFF \ar[r]_{g_2} & \FFF}
 $$
 The composition of these two liftings gives us a lifting of the first square.
\itemt At this point, we have proven the existence of the model structure on the category of curved conilpotent cooperads. Obviously, the adjunction $\Omega_u \dashv B_c$ is a Quillen adjunction. It is a Quillen equivalence by Proposition \ref{prop:eq1}.
\end{itemize}
\end{proof}

\subsection{Fibrations}

\begin{prop}\label{prop:fibrations}
 The fibrations are the retracts of pullbacks of maps of the form $B_c (f): B_c \PPP \ra B_c \QQQ$ where $f: \PPP \ra \QQQ$ is a surjection of operads.
\end{prop}

\begin{lemma}\label{lemma:fibsurj}
 A fibration of curved conilpotent cooperads is surjective.
\end{lemma}

\begin{proof}
 Let $g: \CCC \ra \DDD$ be a fibration of curved conilpotent cooperads. Using the fact that $D^0$ is a dg cocommutative counital coalgebra, one can build the structure of a curved conilpotent cooperad on $\EEE:=\II \oplus  D^0 \otimes \ov \DDD = \II \oplus \ov \DDD \oplus s^{-1}\ov \DDD$; the decomposition $\Delta$ is defined as follows.
 $$
\begin{cases}
 \Delta x := \Delta_\DDD x, \ \text{if } x\in \ov \DDD \subset \ov \EEE\ ,\\
 \Delta s^{-1} x :=(s^{-1}\circ Id + Id \circ' s^{-1}) \Delta x\ . 
\end{cases}
 $$
The coderivation sends $x$ to $dx +s^{-1}x$ and $s^{-1}x$ to $- s^{-1}dx$. Moreover, the curvature is the map $\ov \EEE =D^0 \otimes \ov \DDD \to \DDD \xrightarrow{\theta_\DDD} \mbk$. One can check that this defines a curved conilpotent cooperad and that the maps
\[
	\ov \EEE(n) = D^0 \otimes \ov \DDD(n) \xrightarrow{\epsilon \otimes Id} \ov \DDD(n)\ , n \in \mbn
\]
form a morphism of curved conilpotent cooperads from $\EEE$ to $\DDD$. Consider the following square.
 $$
 \xymatrix{\II \ar[r] \ar[d] & \CCC \ar[d]^g\\ \EEE \ar[r] & \DDD}
 $$
Using the coradical filtration on $\EEE$ and $\DDD$, we find that $G_n^{rad}\EEE \simeq D^0 \otimes G_n^{rad}\DDD$ for any integer $n \geq 1$. So the morphism $\II \ra \EEE$ is a filtered quasi-isomorphism and so a weak equivalence. Since this is also  an injection, then it is an acyclic cofibration. So, the square has a lifting. Subsequently the morphism $\CCC \ra \DDD$ is surjective since $\EEE \to \DDD$ is surjective.
\end{proof}

\begin{proof}[Proof of Proposition \ref{prop:fibrations}]
 It is clear that a retract of a pullback of a map $B_c (f)$, where $f$ is a surjection, is a fibration. Conversely, let $g:\CCC \ra \DDD$ be a fibration of curved conilpotent cooperads. Consider the following diagram
 $$
 \xymatrix{B_c \Omega_u \CCC \ar[rr] && B_c \Omega_u \DDD\\
 \CCC \ar[r] \ar[u] & \EEE \ar[lu] \ar[r] & \DDD \ar[u] }
 $$
 where $\EEE$ is the pullback $B_c \Omega_u \CCC \times_{B_c \Omega_u \DDD} \DDD$. By Lemma \ref{lemma:fibsurj}, $g$ is a surjection and so $\Omega_u (g)$ is also a surjection. So, $B_c \Omega_u (g)$ is a fibration. By the key lemma (Lemma \ref{thm:keylemma}), the morphism $\EEE \ra B_c \Omega_u \CCC$ is a weak equivalence. Since the map $\CCC \ra B_c \Omega_u \CCC$ is an acyclic cofibration, then the map $\CCC \ra \EEE$ is also a weak equivalence and an injection ; that is an acyclic cofibration. Hence, the following diagram has a lifting.
 $$
 \xymatrix{\CCC \ar[r]^= \ar[d] & \CCC \ar[d]^g \\ \EEE \ar[r] &\DDD}
 $$
 So $g$ is a retract of the morphism $\EEE \ra \DDD$.
\end{proof}

\section{Curved conilpotent cooperads as models for homotopy operads}

In Section \ref{section:model}, we have transferred the model structure of the category of dg operads to the category of curved conilpotent cooperads along the cobar construction functor in order to obtain a Quillen equivalence. So curved conilpotent cooperads encode as well the homotopy theory of dg operads. In this section, we make  this statement more concrete; indeed, we show that the cofibrant-fibrant objects of the category of curved conilpotent cooperads correspond to a notion of operads up to homotopy.

\subsection{Homotopy operads}

\begin{defin}[Homotopy operad]
 A \textit{homotopy operad} $\PPP$ is a dg-$\mbs$-module $\PP$ with a distinguished element $1_\PPP \in \PP(1)_0$ together with the data of a curved conilpotent cooperad on $\Tfree^c(s\PP \oplus \mbk \cdot v)$ whose coderivation restricts to $sd_\PP$ on $s\PP$ and such that $dv=s1_\PPP$ and whose curvature $\theta$ is the following map.
\begin{align*}
 \Tfree^c(s\PP \oplus \mbk \cdot v) \twoheadrightarrow s\PP \oplus \mbk \cdot v \twoheadrightarrow \mbk \cdot v &\ra \mbk\\
 v & \mapsto 1
\end{align*}
The curved conilpotent cooperad $\Tfree^c(s\PP \oplus \mbk \cdot v)$ is called the \textit{bar construction} of the homotopy operad $\PPP$ and is denoted $B_c \PPP$. An $\infty$-morphism of homotopy operads from $\PPP$ to $\QQQ$ is a morphism of curved conilpotent cooperads from $B_c \PPP$ to $B_c \QQQ$.
\end{defin}

\begin{nota}
 For any homotopy operad $\PPP=(\PP,\gamma_\PPP,1_\PPP)$, we denote by $B_c^{\leq n}\PPP$ the sub curved conilpotent cooperad of $B_c \PPP$ whose underlying $\mbs$-module is $\Tfree^{\leq n} (s\PP \oplus \mbk \cdot v)$.
\end{nota}

\begin{eg}
The functor bar $B_c$ from the category of operads to the category of curved conilpotent cooperads factorises through an inclusion functor from the category of operads to the category of homotopy operads.
\end{eg}

\begin{prop}\cite[Lemma 2]{LeGrignou16}
 Let $\PP$ be a dg-$\mbs$-module. A structure of homotopy operad on $\PP$ is equivalent to the data of a degree $-1$ map $\gamma: \Tfree^c (s\PP \oplus \mbk \cdot v) \ra s\PP$ which restricts to $sd_\PP$ on $s\PP$ and such that for any tree $T$:
 $$
 \sum_{T' \subset T} \gamma({T/T'})(Id \otimes \cdots \otimes \gamma (T') \otimes \cdots \otimes Id)=(\theta \otimes \pi -\pi \otimes \theta) \Delta_2\ .
 $$
\end{prop}

\begin{prop}
 Let $(\PPP, \gamma_\PPP,1_\PPP)$ and $(\QQQ,\gamma_\QQQ,1_\QQQ)$ be two homotopy operads. There is a one-to-one correspondence between the $\infty$-morphisms from $\PPP$ to $\QQQ$ and the degree $0$ maps $f: \ov \Tfree (s\PP \oplus \mbk \cdot v) \ra s\QQ \oplus \mbk \cdot v$ such that on any tree $T$:
 $$
 \sum_{T=T_1 \sqcup \cdots \sqcup T_k} \gamma_\QQQ(T/T_1,\ldots,T_k) (f(T_1)\otimes \cdots \otimes f(T_k)) = \sum_{T' \subset T} f(T/T') (Id \otimes \cdots \otimes \gamma_\PPP (T') \otimes \cdots \otimes Id)\ ,
 $$
 and such that $\theta_\QQQ f =\theta_\PPP$. Subsequently, $\infty$-morphisms are also equivalent to maps $f: \ov \Tfree (s\PP \oplus \mbk \cdot v) \ra s\QQ $ such that $f+ \theta(-)v$ satisfies the equation above.
\end{prop}

\begin{proof}
 The proof relies on the same techniques as the proof of \cite[10.5.5]{LodayVallette12}.
\end{proof}

\begin{defin}[Infinity-quasi-isomorphisms]
 Let $\PPP$ and $\QQQ$ be two homotopy operads. Let $f: \ov \Tfree (s\PP\oplus \mbk \cdot v) \ra s\QQ$ be an $\infty$-morphism from $\PPP$ to $\QQQ$. We say that $f$ is an $\infty$-\textit{isomorphism} (resp. $\infty$-\textit{monomorphism}, $\infty$-\textit{epimorphism}, $\infty$-\textit{quasi-isomorphism}, $\infty$-\textit{isotopy}) if the restriction $f_{|s\PP}$ of $f$ on $s\PP$ is an isomorphism (resp. monomorphism, epimorphism, quasi-isomorphism, the identity of the $\mbs$-module $s\PP$). An $\infty$-morphism $f: \ov \Tfree (s\PP \oplus \mbk \cdot v) \ra s\QQ$ is said to be strict if $f(T)$ is zero on trees with two vertices or more and if $f(v)=0$.
\end{defin}

\begin{eg}
Let $\PPP$ and $\QQQ$ be two operads considered as homotopy operads. Morphisms of operads from $\PPP$ to $\QQQ$ are exactly strict $\infty$-morphisms.
\end{eg}

\begin{prop}\label{prop:inftymorph}
 An $\infty$-morphism is a monomorphism (resp. isomorphism) if and only if it is an $\infty$-monomorphism (resp. $\infty$-isomorphism)
\end{prop}

\begin{proof}
 The fact that an $\infty$-morphism is a monomorphism if and only if it is an $\infty$-monomorphism follows from a straightforward induction. A similar induction shows that an $\infty$-monomorphism is an isomorphism if and only if it is an $\infty$-isomorphism. So $\infty$-isomorphisms are isomorphisms.
\end{proof}

\begin{prop}\label{prop:isofib}
Let $\CCC=(\CC,\Delta, \epsilon, 1 , \theta)$ be a curved conilpotent cooperad whose underlying graded cooperad is cofree cogenerated by a graded $\mbs$-module $\VV$; that is $\CCC \simeq \Tfree^c \VV$ in the category of graded cooperads. Suppose that there exists $v \in  \VV (1)_2$ such that  $\theta (v)=1$. Then, $\CCC$ is isomorphic to the bar construction of a homotopy operad.
\end{prop}

\begin{proof}
Let $s\PP \subset \VV$ be the kernel of the restriction of the curvature $\theta$ to $\VV$. We have an isomorphism of dg $\mbs$-modules $f_1:\VV \simeq s\PP \oplus \mbk \cdot v$. Consider, the following morphism
\begin{align*}
 f:\Tfree \VV &\to s\PP \oplus \mbk v\\
 x \in \VV &\mapsto f_1(x)\\
 x_1 \otimes \cdots \otimes x_n &\mapsto \theta(x_1 \otimes \cdots \otimes x_n) v\ .
\end{align*}
It induces an isomorphism of graded conilpotent cooperads between $\CCC$ and $\Tfree (s\PP \oplus \mbk \cdot v)$. Let us endow $\Tfree (s\PP \oplus \mbk \cdot v)$ with the structure of curved cooperad obtained by transfer of the coderivation of $\CCC$ and of the curvature of $\CCC$ along this isomorphism. Then, $\Tfree (s\PP \oplus \mbk \cdot v)$ becomes the bar construction of a homotopy operad.
\end{proof}

\begin{prop}\label{prop:strictify}
Let $f$ be an $\infty$-epimorphism (resp. $\infty$-monomorphism) from $\PPP$ to $\QQQ$. There exists an $\infty$-isotopy $g$ such that $fg$ (resp. $gf$) is a strict morphism. 
\end{prop}

\begin{proof}
 The proof relies on the same arguments as \cite[1.3.3.3]{LefevreHasegawa03}.
\end{proof}

\subsection{Obstruction theory of homotopy operads and $\infty$-morphisms}

\begin{prop}\label{lemma:cycle}
 Let $\PPP=(\PP,\gamma_\PPP,1_\PPP)$ and $\QQQ= (\QQ,\gamma_\QQQ,1_\QQQ)$ be two homotopy operads. Let $l$ be a map from $B_c^{\leq n-1}\PPP $ to $s\QQ \oplus \mbk \cdot v$ which can be extended to a morphism of curved conilpotent cooperads from $B_c^{\leq n-1} \PPP$ to $B_c^{\leq n-1}\QQQ$. Let $m$ be the degree $-1$ map from $\Tfree^n(s\PP \oplus \mbk \cdot v)$ to $s\QQ\oplus\mbk \cdot v$ defined on any tree $T$ with $n$ vertices by
$$
m:= \sum_{T' \subset T \atop \# T' \geq 2} l(T/T')(Id \otimes \cdots \otimes \gamma_\PPP (T') \otimes \cdots \otimes Id) - \sum_{T=T_1 \sqcup \cdots \sqcup T_k \atop k \geq 2} \gamma_\QQQ (T/T_1\sqcup \cdots \sqcup T_k) (l (T_1) \otimes \cdots \otimes l (T_k)) \ .
$$
Then, $m$ is a cycle of the chain complex $[\Tfree^n(s\PP \oplus \mbk \cdot v),s\QQ\oplus \mbk \cdot v]$ whose differential is induced by the differential of $s\PP \oplus \mbk \cdot v$ and by the differential of $s\QQ\oplus \mbk \cdot v$.
\end{prop}

\begin{proof}
Let us extend $l$ on $\Tfree^n(s\PP \oplus \mbk \cdot v)$ by $0$. Then, let $L$ be the morphism of cooperads from $B_c^{\leq n} \PPP$ to $B_c^{\leq n}\QQQ$ induced by the map $l$; that is 
$$
L(T) := \sum_{T=T_1 \sqcup \cdots \sqcup T_k \atop k \geq 1} l (T_1) \otimes \cdots \otimes l (T_k)\ .
$$
Notice that $L$ commutes with the coderivations when restricted to $B_c^{\leq n-1} \PPP$. Moreover, let $M$ be the map from $\Tfree^{n} (s \PP \oplus \mbk \cdot v)$ to $B_c \QQQ$ defined as follows on any tree $T$ with $n$ vertices:
\[
M:= LD -DL\ .
\]
where $D$ denotes the coderivation of either $B_c \PPP$ or $B_c \QQQ$. If $d$ denotes the differential of $\Tfree (s \PP \oplus \mbk \cdot v)$ induced by the differential of $s \PP \oplus \mbk \cdot v$. Then
\begin{align*}
DM + M d &= DLD-D^2L+LDd-DLd\\
&=DL(D-d)-D^2L+LDd\\
&=LD(D-d)-D^2L+LDd\\
&=LD^2-D^2L\\
&= (\theta \otimes L - L \otimes \theta)\Delta_2 - (\theta \otimes L - L \otimes \theta)\Delta_2 = 0\ .
\end{align*}
Moreover, let $\pi_{> 1} M$ be the projection of $M$ on $\Tfree^{> 1} (s \QQQ \oplus \mbk \cdot v)$. We have:
 $$
 \pi_{>1} M := \sum_{T=T_1 \sqcup \cdots \sqcup T_k \atop k \geq 2} \sum_i (l(T_1) \otimes \cdots \otimes l(T_i)D \otimes \cdots \otimes  l(T_k)) - \sum_{T=T_1 \sqcup \cdots \sqcup T_k \atop k \geq 2} \pi_{>1} D (l(T_1) \otimes \cdots \otimes l(T_k))
 $$
 Since
\begin{align*}
 \sum_{T=T_1 \sqcup \cdots \sqcup T_k \atop k \geq 2} \pi_{>1} D (l(T_1) \otimes \cdots \otimes l(T_k)) & = \sum_{T=T_1 \sqcup \cdots \sqcup T_k \atop k \geq 2} \sum_{T' \subsetneq  T/T_1, \ldots, T_k}  (Id \otimes \cdots \otimes \gamma_\QQQ (T') \otimes \cdots \otimes Id) (l(T_1) \otimes \cdots \otimes l(T_k))\\
 &= \sum_{T' \subsetneq T} \sum_{T=T_1 \sqcup \cdots \sqcup T' \sqcup \ldots \sqcup T_k \atop k \geq 2}  (l(T_1) \otimes \cdots \otimes \gamma_\QQQ L (T') \otimes \cdots \otimes  l(T_k))\ ,
\end{align*}
then $\pi_{>1} M = 0$. So $M = m$ and so $\partial m = DM + M d =0$.
\end{proof}

\begin{prop}\label{prop:cycle2}
 Let $\PP$ be a graded $\mbs$-module together with a degree $-1$ map $\gamma: \Tfree^{\leq n-1} (s\PP \oplus \mbk \cdot v) \to s\PP$ such that on any tree $T$ with $n-1$ or less vertices:
 $$
 \sum_{T' \subset T} \gamma(T/T') (Id \otimes \cdots \otimes \gamma(T') \otimes \cdots \otimes Id) = (\theta \otimes \pi - \pi \otimes \theta)\Delta_2\ ,
 $$
 where $\theta$ and $\pi$ are defined in the obvious way. In particular, $\gamma$ extends a differential $d$ on $s\PP \oplus \mbk \cdot v$ whose image lies in $s\PP$. Let $\kappa$ be the degree $-2$ map from $\Tfree^n (s\PP \oplus \mbk \cdot v)$ to $s\PP$ defined on any tree $T$ with $n$ vertices by
 $$
 \kappa := \sum_{T' \subsetneq T \atop \# \mathrm{vert}(T') \geq 2} \gamma(T/T') (Id \otimes \cdots \otimes \gamma(T') \otimes \cdots \otimes Id) - (\theta \otimes \pi - \pi \otimes \theta)\Delta_2\ .
 $$ 
 Then $\kappa$ is a cycle of the chain complex $[\Tfree^n(s\PP \oplus \mbk \cdot v), s\PP]$ whose differential is induced by the differential of $s\PP \oplus \mbk \cdot v$ and by the differential of $s\PP$.
\end{prop}

\begin{proof}
 If $n=2$, then $\kappa =- (\theta \otimes \pi - \pi \otimes \theta)\Delta_2$ is a cycle. For $n \geq 3$, we use the same techniques as in the proof of Proposition \ref{lemma:cycle} to prove that it is a cycle.
\end{proof}

The following proposition is a consequence of Proposition \ref{lemma:cycle} and will allow us to show that bar constructions of homotopy operads are fibrant curved conilpotent cooperads.

\begin{prop}\label{prop:lifting}
 Consider the following commutative square of homotopy operads with $\infty$-morphisms
 $$
 \xymatrix{\PPP \ar[r]^u \ar[d]_f & \QQQ \ar[d]^g\\
  \PPP'  \ar[r]_v & \QQQ'\ ,}
 $$
 where $f$ is both an $\infty$-quasi-isomorphism and an $\infty$-monomorphism and $g$ is an $\infty$-epimorphism. Then, this square has a lifting.
\end{prop}

\begin{proof}
 By Proposition \ref{prop:strictify}, we can suppose that $f$ and $g$ are strict morphisms. We will build by induction maps
 $$
 l_n: \Tfree^n (s\PP' \oplus \mbk \cdot  v) \ra s\QQ \oplus \mbk \cdot  v \ , n \geq 1\ ,
$$
 such that
 $$
\begin{cases}
 \partial (l_n) = m_n\ ,\\
 g_1 l_n = v_n \ ,\\
  l_n (f_1 \otimes \cdots \otimes f_1) = u_n \ ,
\end{cases}
  $$
where 
  $$
m_n:=   \sum_{T' \subset T \atop \#vert (T') \geq 2} l_{< n} (Id \otimes \cdots \otimes \gamma_{T'} \otimes \cdots \otimes Id) - \sum_{T=T_1 \sqcup \cdots \sqcup T_k \atop k \geq 2} \gamma (l_{T_1} \otimes \cdots \otimes l_{T_k})\ .
  $$
 Suppose that we have constructed $l_1$, \ldots, $l_{n-1}$. Since, by Proposition \ref{lemma:cycle}, $m_n$ is a cycle in the chain complex $[\Tfree^{n} (s\PP' \oplus v) , s\QQ \oplus \mbk \cdot  v]$ (whose differential is induced by the differential of $s\PP' \oplus \mbk \cdot  v$ and the differential of $s\QQ \oplus \mbk \cdot  v$), constructing $l_n$ amounts to lift the following square.
 $$
 \xymatrix{S^{-1} \ar[r]^{m_n} \ar[d] & [\Tfree^{n} (s\PP' \oplus v) , s\QQ \oplus v] \ar[d]\\
 D^0 \ar[r]_(0.15){(u_n,v_n)} & [\Tfree^{n} (s\PP \oplus v) , s\QQ \oplus v] \times_{[\Tfree^{n} (s\PP \oplus v) , s\QQ' \oplus v]} [\Tfree^{n} (s\PP' \oplus v) , s\QQ' \oplus v]}
 $$
Since $g_1$ is a fibration and since $f_1$ is an acyclic cofibration of dg $\mbs$-modules, then the right vertical map is an acyclic fibration of chain complexes. So the square has a lifting. Thus, we obtain $l_n$. Then, let $L:B_c \PPP' \ra B_c \QQQ$ the morphism of graded cooperads induced by the maps $(l_k)_{k=1}^\infty$. It is a morphism of curved conilpotent cooperads since it commutes with coderivations and since $\theta_{B_c \PPP'} = \theta_{B_c \QQQ'} B_c (v) = \theta_{B_c \PPP'} B_c (g) L = \theta_{B_c \QQQ} L$.
\end{proof}

\subsection{Fibrant curved conilpotent cooperads}

\begin{prop}\label{prop:fibrant}
 The fibrant curved conilpotent cooperads are the curved conilpotent cooperads isomorphic to the bar construction $B_c \PPP$ of a homotopy operad $\PPP$.
\end{prop}

The proof of this proposition consists in showing that a retract of a cofree graded conilpotent cooperad is cofree.

\begin{lemma}\label{lemma:retract}
A retract of a cofree graded conilpotent cooperad $\Tfree^c \VV$ is isomorphic to a cofree graded conilpotent cooperad.
\end{lemma}

\begin{proof}
 Let $\CCC$ be a retract of the cofree curved conilpotent cooperad $\Tfree^c (\VV)$. Let us denote $\WW = F_1\CCC/F_0\CCC$. First, for any integer $n$, the map 
 \[
 F_n \CCC \xrightarrow{\delta} \Tfree (\CCC) \twoheadrightarrow \Tfree^n(\CCC)
 \]
 factorises through $\Tfree^n (\WW)$. Then, consider the following retract diagram
 \[
\begin{tikzcd}
  G_n\CCC \arrow[r] \arrow[d] & G_n \Tfree(\VV) \arrow[r] \arrow[d] & G_n \CCC \arrow[d]\\
\Tfree^n (\WW) \arrow[r] & \Tfree^n (\VV) \arrow[r] & \Tfree^n (\WW)\ .
\end{tikzcd}
 \]
Since the middle vertical arrow is an isomorphism, then the map $G_n \CCC \ra \Tfree^n(\WW )$ is also an isomorphism. Besides, the image through the morphism $\Tfree(\VV) \to \CCC$ of $\VV$ is contained in $\WW$. So, using the projection $\pi:\Tfree(\VV) \to \VV$, one obtains the following map
\[
\CCC \to \Tfree(\VV) \to \VV \to \WW\ ,
\]
and hence one obtains a morphism of cooperads $\CCC \to \Tfree^c (\WW)$. Notice that the composite map $\WW \to \CCC \to \Tfree (\WW)$ is the usual inclusion of $\WW$ into $\Tfree (\WW)$. Finally, consider the following diagram.
\[
\begin{tikzcd}
  G_n\CCC \arrow[r] \arrow[d] & G_n \Tfree(\WW)  \arrow[d] \\
\Tfree^n (\WW) \arrow[r] & \Tfree^n (\WW) 
\end{tikzcd}
 \]
 Since the two vertical maps and the bottom horizontal map are isomorphisms, then so is the map $  G_n\CCC \to  G_n \Tfree(\WW) $. We conclude by Theorem \ref{maclane-homology}.
\end{proof}

\begin{proof}[Proof of Proposition \ref{prop:fibrant}]
 Let $\CCC$ be a fibrant curved conilpotent cooperad. Since the map $\CCC \ra B_c \Omega_u \CCC$ is an acyclic cofibration, it has a right inverse $p$ and so, $\CCC$ is a retract of $B_c \Omega_u \CCC$. So, by Lemma \ref{lemma:retract}, $\CCC$ is cofree: $\CCC:= \Tfree (\VV)$.  Moreover, $p(v)$ is an element of $\VV$ such that $\theta (p(v)) =1$. So, by Proposition \ref{prop:isofib}, $\CCC$ is isomorphic to the bar construction of a homotopy operad. Conversely let $\PPP$ be a homotopy operad. The canonical morphism $B_c \PPP \ra B_c \Omega_u B_c \PPP$ is an $\infty$-monomorphism and an $\infty$-quasi-isomorphism by a variant of Proposition \ref{prop:bar-cobarres}. So, by Proposition \ref{prop:lifting}, it has a left inverse; so $B_c \PPP$ is a retract of  $B_c \Omega_u B_c \PPP$. Since $B_c \Omega_u B_c \PPP$ is fibrant, then $B_c \PPP$ is fibrant.
\end{proof}

\begin{prop}
An $\infty$-morphism of homotopy operads is a cofibration (resp. a fibration, a weak equivalence) of curved conilpotent cooperads if and only if it is an $\infty$-monomorphism (resp. $\infty$-epimorphism, $\infty$-quasi-isomorphism).
\end{prop}

\begin{proof}
 We have already proven (Proposition \ref{prop:inftymorph}) that an $\infty$-morphism is a monomorphism (that is a cofibration) if and only if it is an $\infty$-monomorphism. Let $f: \PPP \ra \QQQ$ be an $\infty$-morphism of homotopy operads. Consider the following square of $\mbs$-modules.
 $$
 \xymatrix{\Omega_u B_c\PPP \ar[r]^{\Omega_u f} & \Omega_u B_c \QQQ \\
 \PPP \ar[u] \ar[r]_{f_1} & \QQQ \ar[u]}
 $$
where $f_1$ is the restriction of $f$ to $\PPP$. The two vertical maps are quasi-isomorphisms by a variant of Proposition \ref{prop:bar-cobarres}. So the lower horizontal map is a quasi-isomorphism if and only if the upper horizontal map is a quasi-isomorphism; that is $f$ is a weak equivalence if and only if $f_0$ is a quasi-isomorphism. Finally, suppose that $f$ is an $\infty$-epimorphism. Since it is surjective, then $\Omega_u (f)$ is surjective and so $B_c \Omega_u (f)$ is a fibration. Let us show that $f$ is a retract of $B_c \Omega_u (f)$. We already know (Proposition \ref{prop:fibrant}) that  $B_c\QQQ$ is a retract of $B_c \Omega_u B_c \QQQ$. Consider the following diagram.
$$
\xymatrix{B_c \PPP \ar[r] \ar[d]_f & B_c \Omega_u B_c \PPP \ar[d]^{B_c\Omega_u f}  \ar@{-->}[r] & B_c\PPP \ar[d]^f \\
 B_c \QQQ\ar[r] & B_c \Omega_u B_c \QQQ \ar[r] & B_c \QQQ}
$$
Finding a morphism $ B_c \Omega_u B_c \PPP \ra B_c \PPP$ making the diagram commute and such that the upper horizontal composite map is the identity amounts to lift the following square, which is possible by Proposition \ref{prop:lifting}.
$$
\xymatrix{B_c \PPP \ar[r]^= \ar[d] & B_c \PPP \ar[d]^{ f} \\
B_c \Omega_u B_c \PPP  \ar[r] & B_c\QQQ} 
$$
Conversely, suppose that $f$ is a fibration. It is an $\infty$-epimorphism because the following diagram of curved conilpotent cooperads has a lifting
$$
\xymatrix{\II \ar[r] \ar[d] & B_c \PPP \ar[d]^{ f} \\
\II \oplus s\QQ \oplus  s^{-1}s\QQ   \ar[r] & B_c\QQQ\ ,} 
$$
where $\II \oplus s\QQ \oplus  s^{-1}s\QQ$ is a dg $\mbs$-module considered as a curved conilpotent cooperad with trivial decomposition $\Delta$.
\end{proof}

\begin{prop}
 Let $\PPP$ and $\QQQ$ be two dg operads. They are linked by a zig-zag of quasi-isomorphisms of dg operads if and only if they are linked by an $\infty$-quasi-isomorphism.
\end{prop}

\begin{proof}
 Suppose that $\PPP$ and $\QQQ$ are linked by an $\infty$-quasi-isomorphism $f$. Then, they are linked by a zig-zag of quasi-isomorphisms of operads as follows.
 $$
 \xymatrix{\PPP & \Omega_u B_c \PPP \ar[l] \ar[r]^{\Omega_u f} & \Omega_u B_c \QQQ \ar[r] &\QQQ}
 $$
 Conversely, suppose that $\PPP$ and $\QQQ$ are linked by a zig-zag of quasi-isomorphisms of operads. Any quasi-isomorphism of operads has an homotopy inverse which is an $\infty$-quasi-isomorphism. So there exists an $\infty$-quasi-isomorphism from $\PPP$ to $\QQQ$.
\end{proof}

\subsection{Homotopy transfer theorem}

Consider an acyclic fibration of dg $\mbs$-modules $p: \PP \ra \QQ$. 

\begin{thm}\label{thm:htt}
 Suppose that $\PP$ has a structure of homotopy operad denoted by $\PPP$. Then, there exists an $\infty$-isotopy $f:\PPP \ra \PPP'$ of homotopy operads and a structure of homotopy operad on $\QQ$ such that the map $p: \PP' \ra \QQ$ is a morphism of homotopy operads.
\end{thm}

\begin{proof}
We build by induction this $\infty$-isotopy and this structure of homotopy operad on $\QQ$; that is we build by induction maps
$$
\begin{cases}
 \gamma_n : \Tfree^n (s\QQ \oplus \mbk \cdot v) \xrightarrow{|-1|} s\QQ\ ,\\
 f_n : \Tfree^n (s\PP \oplus \mbk \cdot v) \ra s\PP\ ,
\end{cases}
$$ 
for $n \geq 2$ such that on any tree $T$ with $n$ vertices: 
$$
\begin{cases}
 \partial \gamma_n =- \sum_{T' \subsetneq T\atop \# vert (T') \geq 2} \gamma_{<n} (Id \otimes \cdots \otimes \gamma_{<n} (T') \otimes \cdots \otimes Id) 
 + (\theta\otimes \pi-\pi\otimes \theta)\Delta_2\ ,\\
 \gamma_n p^{\otimes n} + p \partial (f_n) = \sum_{T' \subseteq T\atop  \# vert (T') \geq 2} pf_{<n} (Id \otimes \cdots \otimes \gamma_{\PPP} (T') \otimes \cdots \otimes Id)\\ - \sum_{T= T_1 \sqcup \ldots \sqcup T_k \atop 1< k <n} \gamma_{<n} (pf_{<n}(T_1) \otimes \cdots \otimes pf_{<n}(T_k))\ ,
\end{cases}
$$
where $p$ is extended to $\mbk\cdot v$ by $p(v)=v$. Suppose that we have built $\gamma_2$, $f_2$, \ldots, $\gamma_{n-1}$, $f_{n-1}$. Consider the chain complex
$$
[\Tfree^n(s\QQ \oplus v), s\QQ] \oplus  [\Tfree^n(s\PP \oplus v), s\PP] \oplus s^{-1}[\Tfree^n(s\PP \oplus v), s\PP]
$$
where $[\Tfree^n(s\QQ \oplus v), s\QQ]$ is endowed with the differential induced by the differential of $s\QQ \oplus v$, such that $dv=p(s1_\PP)$, and the differential of $s\QQ$. Moreover, the differential on the other summands is the adding of $s^{-1}$ to any element of $[\Tfree^n(s\PP \oplus v), s\PP]$. The following morphisms of chain complexes
\begin{align*}
  [\Tfree^n(s\PP \oplus v), s\PP] \oplus s^{-1}[\Tfree^n(s\PP \oplus v), s\PP] &\xrightarrow{Id \oplus \partial} [\Tfree^n(s\PP \oplus v), s\PP] \xrightarrow{p} [\Tfree^n(s\PP \oplus v), s\QQ]\\
 [\Tfree^n(s\QQ \oplus v), s\QQ] \xrightarrow{[p^\otimes,Id] } [\Tfree^n(s\PP \oplus v), s\QQ]
\end{align*}
are respectively a surjection and a weak equivalence. Then, the morphism
$$
[\Tfree^n(s\QQ \oplus v), s\QQ] \oplus  [\Tfree^n(s\PP \oplus v), s\PP] \oplus s^{-1}[\Tfree^n(s\PP \oplus v), s\PP] \ra [\Tfree^n(s\PP \oplus v), s\QQ]
$$
is an acyclic fibration. Moreover, by Proposition \ref{prop:cycle2}, the element
$$
\kappa_n := - \sum_{T' \subsetneq T\atop \# vert (T') \geq 2} \gamma_{<n} (Id \otimes \cdots \otimes \gamma_{<n} (T') \otimes \cdots \otimes Id) + (\theta\otimes \pi-\pi\otimes \theta)\Delta_2
$$

is a cycle of the chain complex $[\Tfree^n(s\QQ \oplus v), s\QQ]$. This gives us the following square of chain complexes.
$$
\xymatrix{S^{-2} \ar[r]^(0.15){(\kappa_n,0,0)} \ar[d] & [\Tfree^n(s\QQ \oplus v), s\QQ] \oplus  [\Tfree^n(s\PP \oplus v), s\PP] \oplus s^{-1}[\Tfree^n(s\PP \oplus v), s\PP]  \ar[d]\\
D^{-1} \ar[r]_{\chi_n} & [\Tfree^n(s\PP \oplus v), s\QQ]}
$$
where $\chi_n$ is the following element of $[\Tfree^n(s\PP \oplus v), s\QQ]$:
$$
\chi_n=  \sum_{T' \subset T\atop  \# vert (T') \geq 2} pf_{<n} (Id \otimes \cdots \otimes \gamma_{\PPP} (T') \otimes \cdots \otimes Id) - \sum_{T= T_1 \sqcup \ldots \sqcup T_k \atop 1< k <n} \gamma_{<n} (pf_{<n}(T_1) \otimes \cdots \otimes pf_{<n}(T_k))\ .
$$
Indeed, $\partial (\chi_n) = \kappa_n p^{\otimes n}$ by Lemma \ref{lemma:chi}. This square has a lifting which gives us $\gamma_n$ and $f_n$.
\end{proof}

\begin{lemma}\label{lemma:chi}
 In the proof of Proposition \ref{thm:htt}, we have 
 \[
 \partial (\chi_n) = \kappa_n p^{\otimes n}\ .
 \]
\end{lemma}

\begin{proof}
 Let us denote $pf$ by $g$ and let us extend it to $\Tfree^n (s\PP \oplus \mbk \cdot v)$ by $0$. Moreover, let us extend $\gamma: \Tfree^{\leq n-1} (s\QQ \oplus \mbk \cdot v) \xrightarrow{|-1|} s\QQ$ to $\Tfree^{n} (s\QQ \oplus \mbk \cdot v)$ by $0$. Moreover, we denote respectively by $G$ and $D$ the morphism of cooperads which extends $g$ and the coderivation which extends $D$. Notice that $G$ commutes with the coderivations when restricted to $\Tfree^{\leq n-1} (s\PP \oplus \mbk \cdot v)$. Then,
 \[
 \chi_n =  \pi(G D - D G)\ .
 \]
 Besides, if we denote by $d$ the differential on $\Tfree^{n} (s\QQ \oplus \mbk \cdot v)$, then we have
\begin{align*}
D(G D - D G) +(G D - D G)d &= DGD -D^2G +GDd -DGd\\
&=DG(D-d) -D^2G +GDd\\
&=GD(D-d) +GDd-D^2G\\
&=GD^2 - D^2G\ .
\end{align*}
As in the proof of Proposition \ref{lemma:cycle}, we have
\[
G D - D G = \pi(G D - D G) = gD-\gamma G = \chi_n \ ,
\]
and so $D(G D - D G) +(G D - D G)d = \partial (\chi_n)$. Moreover, one can show that
\[
\pi D^2G = \gamma D G =  (\theta \otimes g - g \otimes \theta)\Delta_2 - \kappa_n p^{\otimes n}\ .
\]
Since $\pi GD^2 =(\theta \otimes g - g \otimes \theta)\Delta_2$, then $\partial (\chi_n)= \kappa_n p^{\otimes n}$.
\end{proof}

This homotopy transfer theorem may for instance be applied to the homology of a homotopy operad. Indeed, a dg $\mbs$-module is linked to its homology by an acyclic fibration.

\begin{prop}
Let $\VV$ be a dg $\mbs$-module and let $H(\VV)$ be its homology. There exists an acyclic fibration of dg $\mbs$-modules from $\VV$ to $H(\VV)$.
\end{prop}

\begin{proof}
Let $Z(\VV)$ be the $\mbs$-module of cycles of $\VV$. Consider the following diagram of $\mbs$-modules.
 $$
 \xymatrix{\VV & Z(\VV) \ar[l] \ar[r] & H(\VV)}
 $$
 Since any graded $\mbk[\mbs_n]$ module is projective, the surjective morphism $Z(\VV) \ra H(\VV)$ has a right inverse. Thus, we obtain an inclusion $H(\VV) \ra \VV$ which is a quasi-isomorphism. It has a right inverse which is an acyclic fibration.
\end{proof}

\begin{rmk}
Let $\PPP=(\PP,\gamma,1)$ be a dg operad and let $p$ be an acyclic fibration of $\mbs$-modules from $\PP$ to its homology $H(\PP)$. The homotopy transfer theorem applied to $p$ gives operadic Massey products of $\PPP$. We refer to \cite{Livernet15} for a computation of operadic Massey products of the Swiss cheese operad.
\end{rmk}

\subsection{Path object}

For any integer $n \in \mbn$, let $\Phi[n]$ be the linear dual of the normalized Moore complex of the simplicial set $\Delta_n$. For instance $\Phi[1]$ is as follows:
$$
\begin{cases}
 \Phi[1]_0 := \mbk \cdot (0) \oplus \mbk \cdot (1)\ ,\\
 \Phi[1]_{-1} := \mbk \cdot (01)\ ,\\
 \Phi[1]_n =0\ ,\ n \notin \{-1,0\}\ ,\\
 d (0) = (01)\ ,\\
 d(1)=-(01)\ .
\end{cases}
$$

\begin{prop}
 Let $\PPP= (\PP, \gamma_\PPP,1_\PPP)$ be a homotopy operad. The dg $\mbs$-module $\Phi[1] \otimes \PP$ has a structure of homotopy operad that we denote $\Phi[1]\otimes \PPP$ and which is a path object of the homotopy operad $\PPP$. 
\end{prop}

\begin{proof}
For convenience, we will denote $\Tfree^m(  s\Phi[1] \otimes \PP \oplus \mbk \cdot v)$ by $B_c^m \PP_1$ and $\Tfree^m(s\PP \oplus \mbk \cdot v)$ by $B_c^m \PP$. We build by induction maps
$$
\gamma_m: B_c^m\PP_1 \xrightarrow{|-1|}  \Phi[1] \otimes s\PP
$$ 
such that on any tree $T$ with $m$ or less than $m$ vertices
$$
\partial \gamma (T) +\sum_{T'\subsetneq T\atop \# vert (T') \geq 2} \gamma(T/T') (Id \otimes \cdots \otimes \gamma(T') \otimes \cdots \otimes)= (\theta \otimes \pi - \pi \otimes \theta)\Delta_2
$$
where $\pi$ is the projection of $B_c \PP_1$ onto $s\Phi[1] \otimes \PP$ and $\theta$ is the map $B_c\PP_1 \twoheadrightarrow \mbk \cdot v \ra \mbk$.  Moreover, we require the following equality between maps from $B_c^m\PP_1$ to $s\PP $ (resp. $B_c^m\PP$ to $s\Phi[1] \otimes \PP$):
$$
\gamma_{\PPP}(s\delta_i  \otimes Id_{\PP} \oplus Id_v)^{\otimes m} = (s\delta_i \otimes Id_{\PP})\gamma_{m}\ ,
$$
$$
\gamma_{m}(s\sigma \otimes Id_{\PP} \oplus Id_v)^{\otimes m} = (s\sigma\otimes Id_{\PP}) \gamma_{\PPP}\ ,
$$
for the two face maps $\delta_i: \Phi[1] \ra \mbk$ and for the degeneracy map $\sigma: \mbk \ra \Phi[1]$. Suppose that we have built $\gamma_2$, \ldots, $\gamma_{m-1}$. Using the same techniques as in the proof of Proposition \ref{prop:lifting}, building $\gamma_m$ amounts to find a lift to the following square:
$$
\xymatrix{S^{-1} \ar[r] \ar[d] & [B^m_c\PP_1, s\Phi[1]\otimes \PP] \ar[d]\\ D^{-1} \ar[r] &  [B^m_c\PP, s\Phi[1] \otimes \PP] \times_{[B^m_c\PP, s\PP \oplus s\PP]} [B^m_c\PP_1, s\PP \oplus s\PP]\ .} 
$$
Since the map $B^m_c\PP \ra B^m_c\PP_1$ induced by the degeneracy map $\sigma : \mbk \ra \Phi[1]$ is an acyclic cofibration and since the morphism $s\Phi[1] \otimes \PP  \ra s\PP \oplus s\PP$ induced by the two face maps $\Phi[1] \ra \mbk$ is a fibration, then the right vertical map of the diagram is an acyclic fibration. So the square has a lifting.

\end{proof}

If $\PPP$ is an operad, then we can give a precise description of a path object.

\begin{prop}\label{prop:path}
 Let $\PPP= (\PP, \gamma_\PP, 1_\PPP, d_\PP)$ be an operad. Then, a path object of $B_c\PPP$ in the category of curved conilpotent cooperads is given by the homotopy operad $\Phi[1] \otimes\PPP $ whose unit is $\left((0) + (1)\right) \otimes 1_\PPP $ and whose structure map
 $$
 \gamma : \Tfree(s\Phi[1] \otimes \PP  \oplus \mbk \cdot v) \to s\Phi[1] \otimes \PP
 $$
 is defined as follows. Consider the following tree $X$ labelled by elements of $s\Phi[1] \otimes \PP$.
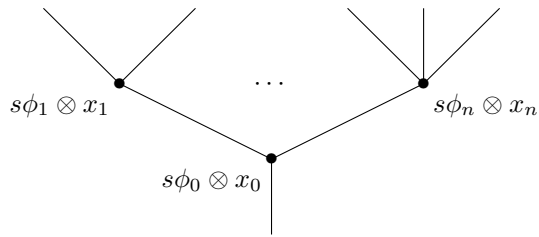
\begin{center}
\begin{tikzpicture}
    \draw (3,0) -- (3,1) ;
    \draw (3,1) node {$\bullet$} ;
    \draw (3,1) node [below left] {$s\phi_0 \otimes x_0$} ;
    \draw (3,1) -- (1,2) ;
    \draw (1,2) node {$\bullet$} ;
    \draw (1,2) node [below left] {$s\phi_1 \otimes x_1$} ;
    \draw (3,2) node {$\cdots$} ;
    \draw (3,1) -- (5,2) ;
    \draw (5,2) node {$\bullet$} ;
    \draw (5,2) node [below right] {$s\phi_n \otimes x_n$} ;
    \draw (1,2) -- (0,3) ;
    \draw (1,2) -- (2,3) ;
    \draw (5,2) -- (4,3) ;
    \draw (5,2) -- (5,3) ;
    \draw (5,2) -- (6,3) ;
\end{tikzpicture}
\end{center}
Then:
\[
\begin{cases}
 \gamma(s\phi_0 \otimes x_0)= -sd \phi_0 \otimes x_0 + (-1)^{|\phi_0|+1} s\phi_0 \otimes d x_0\\
 \gamma(X)= (-1)^{|x_0|} s\phi_0 \otimes \gamma_\PP(x_0 \otimes x_1) \text{ if }n=1\text{ and }\phi_0= \phi_1=(0)\text{ or }\phi_0=\phi_1=(1)\\
 \gamma(X)= (-1)^{|x_0|+1} s(01)\otimes\gamma_\PP(x_0 \otimes x_1) \text{ if }n=1\text{, }\phi_0=(01) \text{ and }\phi_1=(1)\\
 \gamma(X)=  (-1)^{|x_0|+n+1}s(01)\otimes \gamma_\PP(x_0 \otimes_{\mbs_n} x_1 \otimes \cdots \otimes x_n)\text{ if }n\geq 1\text{, }\phi_0=(0) \text{ and }\phi_1,\ldots,\phi_n=(01)\\
 \gamma(v) = s\left((0) + (1)\right) \otimes 1_\PPP\ .
\end{cases}
\]
The map $\gamma$ is zero on all other labelled trees (for instance, the trees of height higher than 2 or the trees which contain $v$ and have at least two vertices).

\end{prop}

\begin{proof}
Let us denote by $D_{\gamma}$ the coderivation of $\Tfree^c(s\PP \otimes \Phi[1] \oplus \mbk \cdot v)$ which extends $\gamma$. We have to prove that 
$$
\begin{cases}
\gamma D_\gamma \left( \left( s\phi_0 \otimes x_0 \right) \otimes v \right) = -s\phi_0 \otimes x_0\ , \\
\gamma D_\gamma \left( v \otimes \left( s\phi_1 \otimes x_1 \right)  \right) = s\phi_1 \otimes x_1\ , \\
\gamma D_\gamma \text{ is zero on all other labelled trees.} 
\end{cases}
$$
The two first equalities are straightforward to prove. Then it is clear that $\gamma D_\gamma$ is zero on any other tree which contains $v$ and on any tree whose height is larger than $3$. Finally, one can easily check that $\gamma D_\gamma$ is zero on any other tree of height one, two or three.
\end{proof}

\subsection{Strict unital homotopy operads}

\begin{defin}[Strict unital homotopy operads]
A strict unital homotopy operad is a homotopy operad $(\PPP, \gamma_\PPP,1)$ such that:
$$
\begin{cases}
\gamma_\PPP (v)= dv =s1_\PPP\ ,\\
\gamma_\PPP (s1_\PPP \otimes sx)= sx\\
\gamma_\PPP (sx \otimes s1_\PPP)= (-1)^{|x|} sx\\
 \gamma_\PPP (sx \otimes v)=  \gamma_\PPP (v \otimes sx)=0\\
\gamma_\PPP (sx \otimes \cdots \otimes s1_\PPP \otimes \cdots \otimes sy)= 0, \ \text{on }\Tfree^{\geq 3} (s\PPP \oplus \mbk \cdot v)\ ,\\
\gamma_\PPP (sx \otimes \cdots \otimes v \otimes \cdots \otimes sy)= 0, \ \text{on }\Tfree^{\geq 3} (s\PPP \oplus \mbk \cdot v)\ .
\end{cases}
$$
Let $\PPP$ and $\QQQ$ be two strict unital homotopy operads. A strict unital $\infty$-morphism from $\PPP$ to $\QQQ$ is an $\infty$-morphism $f: \ov \Tfree (s\PPP \oplus \mbk \cdot v) \ra s\QQQ$ such that
$$
\begin{cases}
f(v)=0\ ,\\
f(sx \otimes \cdots \otimes v \otimes \cdots \otimes sy)=0, \ \text{on }\Tfree^{\geq 2} (s\PPP \oplus \mbk \cdot v)\ ,\\
f(sx \otimes \cdots \otimes s1_\PPP \otimes \cdots \otimes sy)= 0, \ \text{on }\Tfree^{\geq 2} (s\PPP \oplus \mbk \cdot v)\ .
\end{cases}
 $$
 In particular $f(s1_\PPP)=s1_\QQQ$.
\end{defin}

\begin{defin}[Truncated bar construction of a strict unital homotopy operad]
A semi-augmentation of a strict unital homotopy operad $(\PPP, \gamma_\PPP, 1_\PPP)$ is a morphism of graded $\mbs$-modules $\epsilon:\PPP \ra \II$ such that $\epsilon (1_\PPP)=1$. We denote by $\ov \PPP$ the kernel of $\epsilon$ and by $\pi$ the projection of $\PPP$ on $\ov \PPP$ parallel to $1_\PPP$. Let $(\PPP,\gamma_\PPP)$ be a strict unital homotopy operad equipped with a semi augmentation $\epsilon$. The truncated bar construction of $\PPP$ is the conilpotent cooperad $B_r \PPP := \Tfree (s\ov \PPP)$ equipped with the coderivation which extends the map
$$
\ov \gamma_\PPP : \ov \Tfree s\ov \PPP \ra s\ov \PPP
$$
defined by $\ov \gamma_\PPP := \pi \gamma_\PPP$. It is also equipped with the degree $-2$ map $\theta:=\epsilon(s^{-1})\gamma_\PPP$.
\end{defin}

\begin{prop}
 Let $(\PPP,\gamma_\PPP,1, \epsilon)$ be a semi-augmented strict unital homotopy operad. The truncated bar construction $B_r \PPP$ is a curved conilpotent cooperad with curvature $\theta$. Moreover, the composite map
 $$
 B_r \PPP \twoheadrightarrow \ov \PPP \hookrightarrow \PPP
 $$
 induces a morphism of curved conilpotent cooperads from $B_r \PPP$ to $B_c \PPP$. This morphism is universal, in the sense that for any strict unital homotopy operad $\QQQ$, and for any morphism of curved conilpotent cooperads $f:B_r \PPP \ra B_c \QQQ$, there exists a unique strict unital $\infty$-morphism which extends $f$.
 $$
 \xymatrix{B_r \PPP \ar[r]^f \ar[d] &  B_c \QQQ\\
 B_c \PPP \ar@{-->}[ru]}
 $$
\end{prop}

\begin{proof}
 It follows from straightforward calculations.
\end{proof}

\begin{prop}
 Any $\infty$-morphism between strict unital homotopy operads is homotopic to a strict unital $\infty$-morphism.
\end{prop}

\begin{proof}
Let $\PPP$ and $\QQQ$ be strict unital homotopy operads and let $f: B_c \PPP \ra B_c \QQQ$ be a morphism of curved conilpotent cooperads. First choose a semi-augmentation of $\PPP$. Then, denote by $g$ the composite morphism $B_r \PPP \hookrightarrow B_c \PPP \xrightarrow{f} B_c \QQQ$. Let $h: B_c \PPP \ra B_c \QQQ$ be the unique strict unital $\infty$-morphism which extends the morphism $g$. Consider the following square
 $$
 \xymatrix{B_r \PPP \ar[d] \ar[r] & \text{path}(B_c \QQQ) \ar[d] \\
 B_c \PPP \ar[r]_{f,h} & B_c \QQQ \times B_c \QQQ,}
 $$
 where the horizontal upper arrow is the composite morphism $B_r\PPP \xrightarrow{g} B_c \QQQ \ra \text{path}(B_c \QQQ)$. Since the inclusion $B_r \PPP \ra B_c \QQQ$ is an acyclic cofibration and since the map $\text{path}(B_c \QQQ) \ra B_c \QQQ \times B_c \QQQ$ is a fibration (by definition of a path object), then this square has a lifting.
\end{proof}

\begin{prop}
 Let $p:\PP \ra \QQ$ be an acyclic fibration of dg $\mbs$-modules together with a structure of strict unital homotopy operad on $\PP$. Then, there exists a structure of strict unital homotopy operad on $\QQ$ and a strict unital $\infty$-isotopy $f: \PP \ra \PP'$ such that $p: \PP' \to \QQ$ is a strict $\infty$-morphism.
\end{prop}

\begin{proof}
We can impose the strict unital conditions at every steps of the proof of Theorem \ref{thm:htt}.
\end{proof}


\section{Application to algebras over an operad and infinity-morphisms}

In this section, we recall the notions of algebras over an operad, coalgebras over a cooperad and the concept of infinity morphisms between algebras over an operad; see for instance \cite{LodayVallette12} and \cite{LeGrignou16} for more details. Moreover, we give an operadic formulation of these infinity-morphisms.

\subsection{Algebras over an operad,  coalgebras over a cooperad.}

\begin{defin}[Algebra over an operad]
Let $\PPP=(\PP,\gamma,1)$ be an operad. An algebra over $\PPP$ (or for short a $\PPP$-algebra) $\AAA = (\Aa,\psi_\Aa)$ is the data of a chain complex $\Aa$ together with a morphism of operads $\psi_\Aa:\PPP \to \End_\Aa $. 
\end{defin}

Let $f: \Aa \to \BB$ be a morphism of chain complexes. Consider the following pullback of dg $\mbs$-modules.
 $$
 \xymatrix{\End_\Aa \times_{\End^\Aa_\BB}\End_\BB \ar[r] \ar[d] & \End_\Aa \ar[d] \\
 \End_\BB \ar[r] & \End^\Aa_\BB\ ,}
 $$
where the right vertical map and the bottom horizontal map consist respectively in post-composing with $f$ and precomposing with $f^{\otimes n}$.

\begin{lemma}
 The $\mbs$-module $\End_\Aa \times_{\End^\Aa_\BB}\End_\BB$ has a canonical structure of operad induced by the structure on $\End_\Aa$ and the structure on $\End_\BB$.
\end{lemma}

\begin{proof}
Straightforward. 
\end{proof}

\begin{defin}[Morphisms of algebras]
 A morphism of $\PPP$-algebras from $(\Aa,\psi_\Aa)$ to $(\BB,\psi_\BB)$ is the data of a morphism of chain complexes $f: \Aa\to \BB$ such that the following square diagram commutes
 $$
 \xymatrix{\PPP \ar[r]^{\psi_\Aa} \ar[d]^{\psi_\BB} & \End_\Aa \ar[d] \\
 \End_\BB \ar[r] & \End^\Aa_\BB\ .}
 $$
In particular, it corresponds to a morphism of operads from $\PPP$ to $\End_\Aa \times_{\End^\Aa_\BB}\End_\BB$.
\end{defin}

\begin{defin}[Coalgebra over a curved conilpotent cooperad]
 Let $\CCC= (\CC,\Delta,\epsilon, d_\CC,\theta)$ be a curved conilpotent cooperad. A $\CCC$-coalgebra $\DDD=(\DD, \Delta_\DD, d_\DD)$ is the data of a graded $\mbk$-module $\DD$ together with a morphism $\Delta_\DD: \DD \to \CC \circ \DD$ and a degree $-1$ map $d_\DD: \DD \to \DD$ such that
 \[
\begin{cases}
(\epsilon \circ Id)\Delta_\DD=Id\ ,\\
 (\Delta \circ Id) \Delta_\DD =  (Id \circ \Delta_\DD) \Delta_\DD \ ,\\
 \Delta_\DD d_\DD = (d_\CC \circ Id + Id \circ' d_\DD) \Delta_\DD\ ,\\
 d_\DD^2 = (\theta_\CC \circ Id) \Delta_\DD\ .
\end{cases}
 \]
 A morphism of $\CCC$-coalgebras from $\DDD=(\DD, \Delta_\DD, d_\DD)$ to $\EEE=(\EE, \Delta_\EE, d_\EE)$ is a morphism of graded $\mbk$-modules $f: \DD \to \EE$ such that $fd_\DD = d_\EE f$ and $(Id \circ f) \Delta_\DD = \Delta_\EE f$.
\end{defin}

\begin{prop}[\cite{LeGrignou16}]
Let $\Aa$ be a chain complex. The following sets are canonically isomorphic:
\begin{itemize}
\itemt The set of $\Omega_u \CCC$-algebra structures on $\Aa$.
\itemt The set of degree $-1$ maps $\phi_\Aa: \ov\CC \to \End_\Aa$ such that 
\[
\partial(\phi_\Aa) + \gamma_\PPP (\phi_\Aa \otimes \phi_\Aa ) (\Delta_\CCC)_2 = \theta_\CCC (-) Id_\Aa\ .
\]
 \itemt The set of degree $-1$ maps $\gamma_\Aa: \CC \circ \Aa \to \Aa$ such that $(\gamma_\Aa)_{|\Aa}=d_\Aa$ and such that
$$
\gamma_\Aa D_{\gamma_\Aa}  = \theta \circ Id_\Aa\ ,
 $$
 where 
 \[
D_{\gamma_\Aa} = \big( Id \circ (\pi ; \gamma_\Aa)\big)   (\Delta_\CCC \circ Id_\VV) +  d_\CC \circ Id_\VV\ .
 \]

 \itemt The set of degree $-1$ endomorphisms $D$ of the graded $\mbs$-module $\CC \circ \Aa$ such that $D_{|\Aa}=d_\Aa$ and such that $(\CC \circ \Aa, \Delta_\CC \circ Id,D)$ is a $\CCC$-coalgebra.
\end{itemize}
\end{prop}

Therefore, to any $\Omega_u \CCC$-algebra $\AAA=(\Aa,\gamma_\Aa)$, one can associate a $\CCC$-coalgebra $(\CCC \circ \Aa, \Delta_\CC \circ Id, D_\gamma)$. This process is functorial.

\begin{defin}[The bar functor relative to a curved conilpotent cooperad]
 Let $B_\iota$ be the functor from the category of $\Omega_u \CCC$-algebras to the category of $\CCC$-coalgebras which sends $\AAA=(\Aa,\gamma_\Aa)$ to $(\CCC \circ \Aa, \Delta_\CC \circ Id, D_\gamma)$ and sends a morphism $f: \Aa \to \BB$ to the map $Id \circ f : \CC \circ \Aa \to \CC \circ \BB$.
\end{defin}

\subsection{Infinity-morphisms of algebras}

\begin{defin}[Infinity-morphism]
 Let $\AAA=(\Aa, \gamma_\Aa)$ and $\BBB=(\BB, \gamma_\BB)$ be two $\Omega_u \CCC$-algebras. An \textit{infinity-morphism} ($\infty$-morphism for short) from $\Aa$ to $\BB$ is a morphism of $\CCC$-coalgebras from $B_\iota \Aa$ to $B_\iota \BB$.
\end{defin}

These $\infty$-morphisms have a manageable equivalent definition.

\begin{prop}[\cite{LeGrignou16}]
 Let $\AAA=(\Aa, \gamma_\Aa)$ and $\BBB=(\BB, \gamma_\BB)$ be two $\Omega_u \CCC$-algebras. There is a canonical isomorphism between the set $\infty$-morphisms from $ \AAA$ to $\BBB$ and the set of graded maps $f: \CC \circ \Aa \to \BB$ such that
 \begin{equation}\label{eqinftymor}
 f d_{B_\iota \Aa}  = (\gamma_\Aa) (Id \circ f)(\Delta_\CC \circ Id)\ .
\end{equation} 
\end{prop}

\begin{defin}[Infinity-isotopy]
 Let $\AAA=(\Aa, \gamma_\Aa)$ and $\AAA'=(\Aa, \gamma_{\Aa'})$ be two $\Omega_u \CCC$-algebras which have the same underlying chain complex $\Aa$. An infinity-isotopy from $\AAA$ to $\AAA'$ is an $\infty$-morphism $f$ whose first level map is the identity of $\Aa$, that is such that $f_{|\Aa}=Id_\Aa$.
\end{defin}

We give here an other definition of an $\infty$-morphism that will be useful in the sequel.

\begin{lemma}\label{thm:inftymorph}
 An $\infty$-morphism from $(\Aa,\phi_\Aa)$ to $(\BB,\phi_\BB)$ is equivalent to the data of a morphism $g$ of graded $\mbs$-modules from $\CC$ to $\End^\Aa_\BB$ such that
 
\begin{equation}\label{eqinftymor2}
 \partial (g) +\gamma(\phi_\BB \circ g)\Delta + \gamma (g \otimes \phi_\Aa)\Delta_2= 0\ .
\end{equation}
\end{lemma}

\begin{proof}
 An $\infty$-morphism is a map $f:\CC \circ \Aa \to \BB$ satisfying Equation (\ref{eqinftymor}). We have a canonical isomorphism
\[
\hom_{\gMod}(\CC \circ \Aa, \BB) \simeq \hom_{\mathsf{gr}-\mbs-\mathsf{Mod}}(\CC, \End^\Aa_\BB)\ .
\]
Moreover, $f$ satisfies Equation (\ref{eqinftymor}) if and only if its image satisfies Equation (\ref{eqinftymor2}).
\end{proof}

\subsection{Infinity-morphisms of algebras in terms of morphisms of homotopy operads}

\begin{defin}
Let $f:\VV \to \WW$ be a morphism of chain complexes.  We denote by $\PP(\VV,f,\WW)$ the dg $\mbs$-module whose underlying graded $\mbs$-module is 
\[
\PP(\VV,f,\WW) =\End_\VV(n) \oplus s^{-1}\End^\VV_\WW (n) \oplus \End_\WW(n)\ ,
\]
and which is equipped with the following differential
\[
d(g_\VV + s^{-1}g^\VV_\WW + g_\WW) = \partial_\VV(g_\VV) + s^{-1} fg_\VV - s^{-1}\partial^{\VV}_\WW g^\VV_\WW - s^{-1}g_\WW f^{\otimes n} + \partial_\VV(g_\VV)\ ,
\]
where $\partial_\VV$ (resp. $\partial_\WW$, resp. $\partial^\VV_\WW$) is the usual differential of $\End_\VV$ (resp. $\End_\WW$, resp. $\End^\VV_\WW$).
\end{defin}

\begin{lemma}\label{lemma:hpull}
 The square diagram
  $$
 \xymatrix{\PP(\VV,f,\WW) \ar[r] \ar[d] & \End_\VV \ar[d] \\
 \End_\WW \ar[r] & \End^\VV_\WW\ ,}
 $$
  is a homotopy pullback in the model category on $\mbs$-modules.
\end{lemma}

\begin{proof}
One can factorise the map $\End_\VV \to \End^\VV_\WW$ through the sub $\mbs$-module of $\End_\VV(n) \oplus \Phi[1] \otimes\End^\VV_\WW (n) $ made up of the elements $g_\VV + (0)\otimes g_0 +(01)\otimes  g_{01}+(1)\otimes   g_1$ such that $g_0 =f g_\VV$. The first map of this factorisation is an acyclic cofibration and the second one is a fibration. So, the homotopy pullback may be obtained as the pullback of this fibration with the morphism $\End_\WW \to \End^\VV_\WW$. This is exactly $\PP(\VV,f,\WW)$.
\end{proof}

\begin{prop}
 There exists a structure of homotopy operad on $\PP(\VV,f,\WW)$ whose unit is $Id_\VV + Id_\WW$ and whose structure map
 $$
 \gamma : \Tfree(s\PP(\VV,f,\WW) \oplus \mbk \cdot v) \to s\PP(\VV,f,\WW)
 $$
 is defined as follows. Consider a tree as in Proposition \ref{prop:path} labelled by elements of $s\PP(\VV,f,\WW)$. If $n=0$, then $\gamma$ is given by the differential on $\PP(\VV,f,\WW)$. The elements
\[
\begin{cases}
 \gamma(sg_\VV \otimes sg'_\VV)\\
 \gamma(sg_\WW \otimes sg'_\WW)\\
 \gamma(ss^{-1}g^\VV_\WW \otimes sg'_\VV)\\
 \gamma \left(sg_\WW \otimes (ss^{-1}g_1\otimes \cdots \otimes ss^{-1}g_n) \right)\ ,
\end{cases}
\]
are given by the usual composition of morphisms of chain complexes. Notice that in the last case, if one input of $g_\WW$ is not linked to one of the $g_i$, then $\gamma$ acts as if it was linked to $f$. Finally, $\gamma$ sends other labelled trees to zero.
\end{prop}

\begin{proof}
It follows from the same arguments as in the proof of Proposition \ref{prop:path}. 
\end{proof}

\begin{prop}\label{prop:canmorph}
 The projection maps $\PP(\VV,f,\WW) \to \End_\VV$ and $\PP(\VV,f,\WW) \to \End_\WW$ are $\infty$-morphisms of homotopy operads. Moreover the morphism
\begin{align*}
 \End_\VV \times_{\End_\VV^\WW}\End_\WW &\to \PP(\VV,f,\WW)\\
 (x,y) &\mapsto x+ 0 +y
\end{align*}
is an $\infty$-morphism of homotopy operads.
\end{prop}

\begin{proof}
 Straightforward.
\end{proof}

\begin{thm}
Let $\CCC$ be a curved conilpotent cooperad. The data of a morphism of curved cooperads from $\CCC$ to $B_c ( \PP(\VV,f,\WW) )$ is equivalent to the data of $\Omega_u\CCC$-algebra structures on $\VV$ and on $\WW$ together with an $\infty$-morphism from $\VV$ to $\WW$ whose first level map from $\VV$ to $\WW$  is $f$.
\end{thm}

\begin{proof}
Consider a morphism of curved conilpotent cooperads $ \CCC \to B_c(\PP(\VV,f,\WW))$. This is equivalent to the data of a degree $-1$ map
$\phi: \ov\CC \to  \End_\VV(n) \oplus s^{-1}\End^\VV_\WW (n) \oplus \End_\WW(n)$ such that
\begin{equation}\label{eqprophomotopy}
\partial(\phi) + \gamma (\phi \circ \phi)\ov\Delta = \theta(-)(Id_\VV + Id _\WW)\ .
\end{equation}
The map $\phi$ can be decomposed as follows
\[
\phi=  \phi_\VV + s^{-1} \phi_{\WW}^\VV + s \phi_\WW \ .
\]
Then, the above equation (\ref{eqprophomotopy}) is equivalent to the three following equations
\[
\begin{cases}
\partial (\phi_\VV) + \gamma (\phi_\VV \otimes \phi_\VV) \Delta_2 = \theta (-)Id_\VV  \\
\partial (\phi_\WW) + \gamma (\phi_\WW \otimes \phi_\WW) \Delta_2 = \theta (-)Id_\WW  \\
\partial (\phi^\VV_{\WW}) +  f \phi_\VV - \phi_\WW f^{\otimes n} + \gamma (\phi^\VV_{\WW} \otimes \phi_\VV)\Delta_2 + \gamma (\phi_\WW\circ \phi^\VV_{\WW})\ov \Delta=0\ .
\end{cases}
\]
Then $\phi_\VV$ and $\phi_\WW$ are twisting morphisms and so induce morphisms of operads from $\Omega_u\CCC$ to respectively $\End_\VV$ and $\End_\WW$. Moreover, one can extend $\phi_{\WW}^\VV$ to all the $\mbs$-module $\CC$ by sending  $ 1$ to $f$. Then, the last equation rewrites
\[
\partial (\phi^\VV_{\WW}) +  \gamma (\phi^\VV_{\WW} \otimes \phi_\VV)\Delta_2 + \gamma (\phi_\WW\circ \phi^\VV_{\WW})\Delta=0\ .
\]
By Lemma \ref{thm:inftymorph}, $\phi_{\WW}^\VV$ defines an $\infty$-morphism from $(\VV,\phi_\VV)$ to $(\WW,\phi_\WW)$ whose first level map is $f$.
\end{proof}

The next corollary generalises a result of Fresse (\cite{Fresse09ter}) that describes a path in the space of algebraic structures on a chain complex in terms of infinity-isotopy.\\

\begin{cor}
Let $\CCC$ be a curved conilpotent cooperad. The data of a morphism of curved cooperads from $\CCC$ to $B_c ( \Phi[1] \otimes\End_\Aa )$ (as defined in Proposition \ref{prop:path}) is equivalent to the data of two $\Omega_u\CCC$-algebra structures on $\Aa$ and an $\infty$-isotopy between them.
\end{cor}

\begin{proof}
It suffices to notice that  $ \Phi[1] \otimes\End_\Aa \simeq \PP(\Aa,Id_\Aa,\Aa)$.
\end{proof}

\subsection{Homotopy transfer theorem for algebras over an operad}\label{secalg} The homotopy transfer theorem is a result that holds for algebras over any cofibrant operad; see for instance \cite{LodayVallette12} and \cite[Theorem 3.5]{BergerMoerdijk03}. We give here an interpretation of this result in terms of homotopy operads.

\begin{prop}[After \cite{BergerMoerdijk03}]\label{prop:htt}
Let $\CCC$ be a curved conilpotent cooperad. Let $p: \Aa \to \VV$ be an acyclic fibration of chain complexes. Suppose that $\Aa$ is endowed with a structure of $\Omega_u\CCC$-algebra denoted $\gamma$. Then there exists:
 
\begin{itemize}
 \itemt a new structure $\gamma'$ of $\Omega_u \CCC$-algebra on $\Aa$, together with an $\infty$-isotopy $i: (\Aa,\gamma) \to (\Aa,\gamma')$,
\itemt a structure $\gamma_\VV$ of $\Omega_u \CCC$ -algebra on $\VV$ such that $p$ is a morphism of $\Omega_u \CCC$-algebras from $(\Aa,\gamma')$ to $(\VV,\gamma_\VV)$.
\end{itemize}
\end{prop}

\begin{lemma}\label{lemma:hpullback}
 In the context of Proposition \ref{prop:htt}, the morphism $B_c(\End_\Aa \times_{\End^\Aa_\VV} \End_\VV) \to B_c(\PP(\Aa,p,\VV))$ introduced in Proposition \ref{prop:canmorph} is an acyclic cofibration. 
\end{lemma}

\begin{proof}
Since $p$ is a fibration, then the pullback  $\End_\Aa \times_{\End^\Aa_\VV} \End_\VV$ is also a homotopy pullback in the model category of $\mbs$-modules. So by Lemma \ref{lemma:hpull}, the map $\End_\Aa \times_{\End^\Aa_\VV} \End_\VV \to \PP(\Aa,p,\VV)$ is a quasi-isomorphism and so the morphism $B_c(\End_\Aa \times_{\End^\Aa_\VV} \End_\VV) \to B_c(\PP(\Aa,p,\VV))$ is a weak equivalence. Moreover, it is an injection and so a cofibration.
\end{proof}

\begin{proof}[Proof of Proposition \ref{prop:htt}]
The structure of $\Omega_u\CCC$-algebra on $\Aa$ is given by a morphism of curved conilpotent cooperads  $\CCC \to B_c \End_\Aa$. Moreover, since $p$ is an acyclic fibration, then the map $\PP(\Aa,p,\VV) \to \End_\Aa$ is also an acyclic fibration and hence the following diagram has a lifting
 \[
 \xymatrix{\emptyset \ar[r] \ar[d] & B_c \PP(\Aa,p,\VV) \ar[d] \\
 \CCC \ar[r] & B_c \End_\Aa\ .}
 \]
Besides, the map of Lemma \ref{lemma:hpullback} has a left inverse. The following composite map
\[
\CCC \to B_c \PP(\Aa,p,\VV) \to B_c( \End_\Aa \times_{\End^\Aa_\VV} \End_\VV) )
\]
induces a new structure of $\Omega_u\CCC$-algebra on $\Aa$ and a structure of $\Omega_u\CCC$-algebra on $\VV$ such that $p$ is a morphism of $\Omega_u\CCC$-algebras. The following diagram is commutative and has a lifting.
 $$
 \xymatrix{B_c( \End_\Aa \times_{\End^\Aa_\VV} \End_\VV)\ar[r] \ar[d] &  B_c (\End_\Aa) \ar[r] & B_c (\Phi[1]\otimes \End_\Aa) \ar[d] \\
 B_c  \PP(\Aa,p,\VV) \ar[rr] && B_c \End_\Aa \times B_c \End_\Aa}
 $$
 So the new structure of $\Omega_u\CCC$-algebra on $\Aa$ is homotopic to the old one. This corresponds to an $\infty$-isotopy.
\end{proof}


\section*{Appendix A: Colored bar-cobar adjunction}

Consider the adjunction $\Omega_u \dashv B_c$ described above and relating curved conilpotent cooperads to operads. We have shown that the projective model structure on the category of operads may be transferred to the category of curved conilpotent cooperads along this adjunction. In other words, there exists a model structure on the category of  curved conilpotent cooperads whose cofibrations (resp. weak equivalences) are the morphisms whose image under $\Omega_u$ is a cofibration (resp. weak equivalence). In this appendix, we show that this method cannot be extended to the multi-colors framework, that is to dg categories and curved conilpotent cocategories. As an immediate consequence, it cannot be extended to colored operads.

\begin{defin}
 A dg (resp. graded) quiver $(X,\VV)$ is the data of a set of objects $X$ and a chain complex (resp. graded $\mbk$-module) $\VV(x,x')$ for any $(x,x')\in X^2$. A morphism of quivers $F$ from $(X,\VV)$ to $(Y,\WW)$ is the data of a function $F: X\to Y$ and morphisms $F_{x,x'}: \VV(x,x')\to \WW(F(x),F(x'))$.
\end{defin}

\begin{eg}
 For any set $X$, we denote by $I_X$ the quiver whose set of object is $X$ and such that
 \[
\begin{cases}
 I_X(x,y)= 0 \text{ if }x \neq y\ ,\\
 I_X(x,x)=\mbk\ .
\end{cases}
 \]
\end{eg}

\begin{defin}
A differential graded (dg) category $\AAA=(X, \Aa, \gamma, (1_x)_{x \in X})$ is the data of a dg quiver $(X,\Aa)$, an associative composition $\gamma_{x,y,z}: \Aa(x,y) \otimes \Aa(y,z)\to \Aa(x,z)$ together with units $1_x\in \Aa(x,x)_0$ for this composition. 
\end{defin}

\begin{defin}
 A curved conilpotent cocategory $\CCC=(X, \CC, \Delta, d, \theta)$ is the data of a graded quiver $(X,\CC)$, a conilpotent coassociative decomposition $\Delta: \CC(x,z) \to \bigoplus_y \CC(x,y) \otimes \CC(y,z)$ together with a degree $-1$ map $d:\CC(x,y)\to\CC(x,y)$ for any $(x,y)\in X^2$, and degree $-2$ maps $\theta: \CC(x,x)\to \mbk$ such that
\begin{align*}
&\Delta d  = (d \otimes Id + Id \otimes d) \Delta\ ,\\
 & d^2 = (\theta \otimes Id-Id \otimes \theta)\Delta\ ,\\
 &\theta d =0\ . 
\end{align*}
\end{defin}

Curved conilpotent cocategories are related to dg categories by an adjunction \`a la bar cobar that we denote $\Omega_u \dashv B_c$ since it extends the adjunction between unital algebras and curved conilpotent coalgebras that we described in \cite[\S 8.3]{LeGrignou16} and that was already denoted $\Omega_u \dashv B_c$. On the one hand, let $\AAA:=(X,\Aa, \gamma, (1_x)_{x \in X})$ be a dg category. Its bar construction is the curved conilpotent cocategory $B_c \AAA := \Tfree^c (s\Aa \oplus  s^{2}I_X)$. It is equipped with the coderivation which extends the following map.
\begin{align*}
 \Tfree (s\Aa \oplus s^{2}I_X) \twoheadrightarrow \ov \Tfree^{\leq 2} (s\Aa \oplus s^{2}I_X) & \ra s\Aa \oplus s^{2}I_X\\
 sx \otimes sy &\mapsto (-1)^{|x|} s \gamma_\AAA (x \otimes y)\\
 sx \otimes s^2 1_c & \mapsto 0\\
 s^2 1 & \mapsto s1\\\
 sx & \mapsto -sdx\ .
\end{align*}
Its curvature is the degree $-2$ map.
\begin{align*}
 \Tfree (s\Aa \oplus s^{2}I_X) \twoheadrightarrow  s^{2}I_X & \ra I_X\\
s^2 1 &\mapsto 1\ .
 \end{align*}

On the other hand, let $\CCC:=(C,\CC,\Delta,d ,\theta)$ be a curved conilpotent cooperad. Its cobar construction is made up of the graded category
$$
\Omega_u \CCC := \Tfree (s^{-1}  \CC)\ ,
$$
together with the following derivation,
$$
s^{-1}x \mapsto  \theta (x) 1 - s^{-1} dx - \sum (-1)^{|x_1|} s^{-1}x_1 \otimes s^{-1}x_2\ ,
$$
where $\Delta x = \sum x_1 \otimes x_2$.\\

\begin{prop}
The bar construction and the cobar construction are both functors. Moreover, the functor $\Omega_u$ is left adjoint to the functor $B_c$.
\end{prop}

\begin{proof}[Proof]
 The proof relies on the same arguments as the proof of \cite[Proposition 21]{LeGrignou16}.
\end{proof}

Tabuada proved in \cite{Tabuada05} that the category of dg categories may be equipped with a model structure as follows.

\begin{thm}[ \cite{Tabuada05}]
 There exists a model structure on the category of dg categories such that a morphism $F:(X,\Aa, \gamma, (1_x)_{x \in X}) \to (X,'\Aa', \gamma', (1_x)_{x \in X'}) $ is
\begin{itemize}
 \itemt a weak equivalence if and only if the map $F_{x,y}: \Aa(x,y) \to \Aa'(F(x),F(y))$ is a quasi-isomorphism for any $(x,y) \in X^2$ and the functor $H_0(F)$ is an equivalence of categories,
 \itemt a fibration if and only if the map $F_{x,y}: \Aa(x,y) \to \Aa'(F(x),F(y))$ is a degreewise surjection for any $(x,y) \in X^2$ and the functor $H_0(F)$ is an isofibration.
\end{itemize}
\end{thm}

\begin{thm}\label{thm:appendixfinal}
 There does not exist a model structure on the category of curved conilpotent cocategories such that the functor $\Omega_u$ preserves cofibrations and weak equivalences (and so is a left Quillen functor).
\end{thm}

\begin{proof}[Proof]
 Let $I$ be the curved conilpotent cocategory with one object $0$ and such that $I (0;0) := 0$. Then, $\Omega_u (I)$ is the dg category $I_0$. Let $J$ be the dg category with two objects $0$ and $1$ and such that 
 $$
 J(i,j)= \mbk\ ,\forall i,j \in \{0,1\}\ ,
 $$
 with obvious units and composition. It is clear that the functor $J \to I_0$ given by the identity of $\mbk$ is an acyclic fibration of dg categories. If such a model structure exists on the category of curved conilpotent cocategories, then the morphism
 $$
 B_c J \times_{B_c I_0} I \to I
 $$
 is an acyclic fibration and the morphism $\Omega_u  (B_c J \times_{B_c I_0} I) \to \Omega_u I= I_0$ is a weak equivalence of dg categories. By Lemma \ref{lemmaappendixop}, $\Omega_u  (B_c J \times_{B_c I_0} I) =I_{\{0,1\}}$. Since the morphism $I_{\{0,1\}} \to I_0$ is not a weak equivalence, then such a model structure does not exist.
\end{proof}

\begin{lemma}\label{lemmaappendixop}
 The pullback  $B_c J \times_{B_c I_0} I$ of the proof of Theorem \ref{thm:appendixfinal} is the cocategory with two objects $0$ and $1$ and such that
 $$
 B_c J \times_{B_c I_0} I (i,j)= 0 \ ,\ \forall (i,j)\in \{0,1\}^2\ .
 $$
\end{lemma}

\begin{proof}[Proof]
 It is clear that $B_c J \times_{B_c I_0} I$  is the biggest sub cocategory of $B_cJ$ whose image in $B_c I_0$ is in the image of $I$. Then, $F^{rad}_1( B_c J \times_{B_c I_0} I)$ lies inside $F^{rad}_1 B_c J$ and its image in $B_c I_0$ is zero. So, a straightforward checking shows that $F^{rad}_1 (B_c J \times_{B_c I_0} I)$ is zero and hence $ B_c J \times_{B_c I_0} I$ is as described in the lemma.
\end{proof}

\bibliographystyle{amsalpha}
\bibliography{biblg}

\def\cprime{$'$}
\providecommand{\bysame}{\leavevmode\hbox to3em{\hrulefill}\thinspace}
\providecommand{\MR}{\relax\ifhmode\unskip\space\fi MR }
\providecommand{\MRhref}[2]{%
  \href{http://www.ams.org/mathscinet-getitem?mr=#1}{#2}
}
\providecommand{\href}[2]{#2}
\begin{thebibliography}{Tab05}

\bibitem[AC03]{AubryChataur03}
Marc Aubry and David Chataur, \emph{Cooperads and coalgebras as closed model
  categories}, J. Pure Appl. Algebra \textbf{180} (2003), no.~1-2, 1--23.

\bibitem[AR94]{AdamekRosicky94}
Ji{\v{r}}{\'{\i}} Ad{\'a}mek and Ji{\v{r}}{\'{\i}} Rosick{\'y}, \emph{Locally
  presentable and accessible categories}, London Mathematical Society Lecture
  Note Series, vol. 189, Cambridge University Press, Cambridge, 1994.

\bibitem[BM03]{BergerMoerdijk03}
C.~Berger and I.~Moerdijk, \emph{Axiomatic homotopy theory for operads},
  Comment. Math. Helv. \textbf{78} (2003), no.~4, 805--831.

\bibitem[Fre09]{Fresse09ter}
Benoit Fresse, \emph{Operadic cobar constructions, cylinder objects and
  homotopy morphisms of algebras over operads}, Alpine perspectives on
  algebraic topology, Contemp. Math., vol. 504, Amer. Math. Soc., Providence,
  RI, 2009, pp.~125--188.

\bibitem[GJ94]{GetzlerJones94}
E.~Getzler and J.~D.~S. Jones, \emph{{Operads, homotopy algebra and iterated
  integrals for double loop spaces}}, \texttt{hep-th/9403055} (1994).

\bibitem[Gri]{LeGrignou16}
Brice~Le Grignou, \emph{Homotopy theory of unital algebras},
  \texttt{arXiv:1612.02254 }.

\bibitem[Hin97]{Hinich97}
Vladimir Hinich, \emph{Homological algebra of homotopy algebras}, Comm. Algebra
  \textbf{25} (1997), no.~10, 3291--3323.

\bibitem[Hin01]{Hinich01}
\bysame, \emph{D{G} coalgebras as formal stacks}, J. Pure Appl. Algebra
  \textbf{162} (2001), no.~2-3, 209--250.

\bibitem[HM12]{HirshMilles12}
Joseph Hirsh and Joan Mill{\`e}s, \emph{Curved {K}oszul duality theory}, Math.
  Ann. \textbf{354} (2012), no.~4, 1465--1520.

\bibitem[Hov99]{Hovey99}
Mark Hovey, \emph{Model categories}, Mathematical Surveys and Monographs,
  vol.~63, American Mathematical Society, Providence, RI, 1999.

\bibitem[LGL]{LeGrignouLejay19b}
Brice Le~Grignou and Damien Lejay, \emph{Operads, trees and souches}.

\bibitem[LH03]{LefevreHasegawa03}
K.~Lefevre-Hasegawa, \emph{Sur les {A}-infini cat\'egories},
  \texttt{arXiv.org:math/0310337} (2003).

\bibitem[Liv15]{Livernet15}
Muriel Livernet, \emph{Non-formality of the swiss-cheese operad}, J Topology
  \textbf{8} (2015), no.~4, 1156--1166.

\bibitem[LV12]{LodayVallette12}
Jean-Louis Loday and Bruno Vallette, \emph{Algebraic operads}, Grundlehren der
  Mathematischen Wissenschaften [Fundamental Principles of Mathematical
  Sciences], vol. 346, Springer-Verlag, Berlin, 2012.

\bibitem[ML95]{MacLane95}
Saunders Mac~Lane, \emph{Homology}, Classics in Mathematics, Springer-Verlag,
  Berlin, 1995, Reprint of the 1975 edition.

\bibitem[MV09]{MerkulovVallette09I}
Sergei Merkulov and Bruno Vallette, \emph{Deformation theory of representations
  of prop(erad)s. {I}}, J. Reine Angew. Math. \textbf{634} (2009), 51--106.

\bibitem[Pos11]{Positselski11}
Leonid Positselski, \emph{Two kinds of derived categories, {K}oszul duality,
  and comodule-contramodule correspondence}, Mem. Amer. Math. Soc. \textbf{212}
  (2011), no.~996, vi+133.

\bibitem[Spi01]{Spitzweck01}
Markus Spitzweck, \emph{Operads, algebras and modules in general model
  categories}, \texttt{arXiv.org:math/0101102} (2001).

\bibitem[Tab05]{Tabuada05}
G~Tabuada, \emph{Une structure de cat\'egorie de mod\`eles de quillen sur la
  cat\'egorie des dg-cat\'egories.}, C. R. Math. Acad. Sci. Paris \textbf{340}
  (2005), no.~1, 15--19.

\bibitem[Val14]{Vallette14}
B.~Vallette, \emph{Homotopy theory of homotopy algebras}, ArXiv:1411.5533
  (2014).

\end{thebibliography}

\end{document}